%% file: main.tex
\documentclass[reqno]{amsart}
\usepackage{amsfonts}
\usepackage{amsmath}
\usepackage{amsthm, amscd}
\usepackage{mathtools}
\usepackage{amssymb}
\usepackage{mathrsfs}
\usepackage{tikz-cd}

\usepackage[alphabetic]{amsrefs}
\usepackage{etex}
\usepackage{hyperref}
\hypersetup{
  colorlinks = true,
  linkcolor  = black
}
\usepackage{color}

\usepackage{tikz}
\usepackage{xypic}

\allowdisplaybreaks

\theoremstyle{plain}\newtheorem{Theorem}{Theorem}[section]
\theoremstyle{plain}\newtheorem{Corollary}[Theorem]{Corollary}
\theoremstyle{plain}\newtheorem{Lemma}[Theorem]{Lemma}
\theoremstyle{plain}\newtheorem{Definition}[Theorem]{Definition}
\theoremstyle{plain}\newtheorem{Proposition}[Theorem]{Proposition}
\theoremstyle{plain}
\theoremstyle{plain}
\theoremstyle{plain}\newtheorem*{Theorem*}{Theorem}

\makeatletter
\newtheorem*{rep@theorem}{\rep@title}
\newcommand{\newreptheorem}[2]{%
\newenvironment{rep#1}[1]{%
 \def\rep@title{#2 \ref{##1}}%
 \begin{rep@theorem}}%
 {\end{rep@theorem}}}
\makeatother
\theoremstyle{plain}\newreptheorem{theorem}{Theorem}

\theoremstyle{remark}\newtheorem{remark}[Theorem]{Remark}
\theoremstyle{remark}\newtheorem{example}[Theorem]{Example}

\DeclarePairedDelimiter\paren{\lparen}{\rparen}

\numberwithin{equation}{section}

\DeclareMathOperator{\ima}{Im}

\DeclareMathOperator{\Sp}{Sp}
\DeclareMathOperator{\SU}{SU}

\DeclareMathOperator{\SL}{SL}

\DeclareMathOperator{\Stab}{Stab}
\DeclareMathOperator{\Orb}{Orb}

\DeclareMathOperator{\Bif}{Bif}
\DeclareMathOperator{\Ad}{Ad}

\DeclareMathOperator{\Hom}{Hom}
\DeclareMathOperator{\Hess}{Hess}
\DeclareMathOperator{\Tr}{Tr}

\DeclareMathOperator{\id}{id}
\DeclareMathOperator{\ind}{ind}
 
\DeclareMathOperator{\CS}{CS} 
\DeclareMathOperator{\hol}{Hol} 
\DeclareMathOperator{\sym}{Sym}
\DeclareMathOperator{\U}{U}

\DeclareMathOperator{\grad}{grad}

\newcommand{\bC}{\mathbb{C}}
\newcommand{\bH}{\mathbb{H}}
\newcommand{\bK}{\mathbb{K}}

\newcommand{\bR}{\mathbb{R}}
\newcommand{\bZ}{\mathbb{Z}}

\newcommand{\clA}{\mathcal{A}}

\newcommand{\clC}{\mathcal{C}}
\newcommand{\clH}{\mathcal{H}}
\newcommand{\clG}{\mathcal{G}}

\newcommand{\clL}{\mathcal{L}}
\newcommand{\clM}{\mathcal{M}}
\newcommand{\clP}{\mathcal{P}}
\newcommand{\clR}{\mathcal{R}}

\newcommand{\frg}{\mathfrak{g}}

\newenvironment{claim}[1]{\par\noindent{\emph{Claim:}}\space#1}{}

\begin{document}

\author{Shaoyun Bai}
\address{Department of Mathematics, Princeton University, New Jersey 08544, USA}
\email{shaoyunb@math.princeton.edu}
\title{A symplectic formula of generalized Casson invariants}
\begin{abstract}
Suppose $Y$ is an integer homology 3-sphere, 
Taubes \cite{taubes1990casson} proved that the number of irreducible critical orbits of the perturbed Chern-Simons functional on $Y$, counted with signs, is equal to the algebraic intersection number of two character varieties associated with Heegaard splittings when the structure group is $\SU(2)$. Taubes' result established a relationship between gauge theory and the Casson invariant. This article proves an analogous identification result for $\SU(n)$ generalized Casson invariants. As a special case, we show that the $\SU(3)$ Casson invariant of Boden-Herald \cite{boden1998the} can be equivalently calculated by taking an appropriate intersection number of Lagrangian submanifolds.
\end{abstract}

\maketitle

\setcounter{tocdepth}{1}

\section{Introduction}
\input{introduction}

\section{Recap of $3$--dimensional gauge theory}
\label{sec_preliminary}
\input{preliminary}

\section{Extended moduli spaces}
\label{sec_Extended}
\input{Extended_moduli}

\section{Equivariant transversality of Lagrangian submanifolds}
\label{sec_Lag_int}
\input{Lagrangian_intersection}

\section{The gauge-symplectic correspondence}
\label{sec_index}
\input{index}

\bibliographystyle{amsalpha}
\bibliography{references}

\end{document}

%% file: introduction.tex

In 1985, Casson introduced an invariant for integer homology 3-spheres (see \cite{akbulut1990casson, marin1988nouvel}) by considering the $\SU(2)$--representations of the fundamental group. Suppose $Y$ is an integer homology 3-sphere with a Heegaard splitting $Y=H_1\cup_\Sigma H_2$, then the inclusions  of $\Sigma$ in $H_1$ and $H_2$ induce embeddings of the $\SU(2)$--character varieties of $\pi_1(H_1)$ and $\pi_1(H_2)$ into the $\SU(2)$--character variety of $\pi_1(\Sigma)$. The character variety of $Y$ is given by the intersection of the character varieties of $M_1$ and $M_2$ as a set. After taking a sufficiently small generic perturbation on the embeddings, one obtains a well-defined intersection number of the character varieties of $\pi_1(H_1)$ and $\pi_1(H_2)$. The intersection number is always even (see, for example, \cite[Corollary 17.6]{saveliev2011lectures}), and the Casson invariant is defined to be half of this intersection number.

On the other hand, let $P$ be the trivial $\SU(2)$ bundle over $Y$ with a fixed trivialization, and let 
$\clC(Y)$ be the affine space of smooth connections on $P$. Let $\theta$ be the flat connection associated with the given trivialization of $P$. For $B\in\clC(Y)$, suppose $B= \theta + b$ where $b\in \Omega^1(Y,\mathfrak{su}(2))$, then the \emph{Chern-Simons functional}
$$
\CS : \clC(Y) \to \bR
$$
is defined by
$$
\CS(B) = \frac{1}{2} \int_Y \Tr \paren*{ db\wedge b+\frac23 b\wedge b\wedge b }.
$$
After fixing a Riemannian metric on $Y$, the formal gradient of $\CS$ at $B\in \clC(Y)$ is equal to $*F_B$, where $*$ is the Hodge star operator of the given metric; therefore, the critical points of $\CS$ are given by flat connections, and the set of critical orbits of $\CS$ has a one-to-one correspondence with the conjugacy classes of $\SU(2)$--representations of $\pi_1(Y)$. Taubes \cite{taubes1990casson} proved that after adding a generic perturbation to the Chern-Simons functional, the signed count of its critical orbits is equal to two times the Casson invariant. 

A natural question about the Casson invariant is whether it can be extended to more general groups and more general 3-manifolds. This question has been studied intensively since the introduction of the theory, and many constructions are given using either intersection of character varieties arisen from Heegaard splittings or the Chern-Simons functionals. The Heegaard-splitting approach has been studied by Boyer-Nicas \cite{boyer1990varieties}, Walker \cite{walker1992extension}, Cappell-Lee-Miller \cite{cappell1990symplectic}, and Curtis \cite{curtis1994generalized}.
More recently, Abouzaid-Manolescu \cite{abouzaid2017sheaf} studied the intersection of $\SL_2(\bC)$ character varieties arisen from Heegaard splittings and obtained a sheaf-theoretic 3-manifold invariant. On the gauge theory side, Boden-Herald \cite{boden1998the} defined an $\SU(3)$ Casson invariant for integer homology spheres by introducing real-valued correction terms on the reducible critical orbits. Variations of the $\SU(3)$ Casson invariant have also been defined by Boden-Herald-Kirk \cite{BHK2001} and Cappell-Lee-Miller \cite{cappell2002perturbative}. Together with Zhang, the author \cite{bai2020equivariant} generalized the construction of Boden-Herald and defined perturbative $\SU(n)$--Casson invariants for integer homology spheres. The construction of \cite{bai2020equivariant} is reviewed in Section \ref{sec_preliminary}.

This article studies the relationship of the gauge-theoretic generalization of Casson invariants in \cite{bai2020equivariant} with Heegaard splittings. We obtain an identification result analogous to Taubes' theorem \cite{taubes1990casson}: the quantities entering into the definition of the $\SU(n)$--Casson invariant could all be expressed in terms of quantities constructed using finite-dimensional symplectic geometry. As a special case, we show that the $\SU(3)$ Casson invariant of Boden-Herald \cite{boden1998the} is equal to an appropriate equivariant intersection number of character varieties associated with Heegaard splittings.

Now we describe the contents of this paper in more detail. Let $Y=H_1\cup_\Sigma H_2$ be a closed, oriented 3-manifold with a Heegaard splitting, and let $G$ be a compact Lie group with Lie algebra $\frg$.
In general, the $G$--character variety of $\pi_1(\Sigma)$  have singularities. Their appearance put forth difficulties on constructing perturbations and studying intersection theory from the differential-geometric approach. To get around this difficulty, we make use of the \emph{extended moduli space of flat connections} $\clM^\frg(\Sigma')$ on $\Sigma'$ introduced by Jeffrey \cite{jeffrey1994extended} and Huebschmann \cite{huebschmann1995symplectic} to study an equivalent question using equivariant geometry, where $\Sigma'$ be the Riemann surface obtained from $\Sigma$ be removing a disc. Roughly speaking, there is an open subset $\hat{\clM}^{\mathfrak{g}}(\Sigma') \subset \clM^\frg(\Sigma')$ which is a smooth symplectic manifold and there is a Hamiltonian $G$--action on $\hat{\clM}^{\mathfrak{g}}(\Sigma')$ whose moment map reduction $\mu^{-1}(0) / G$ is naturally identified with the $G$--character variety of $\pi_1(\Sigma)$. The $G$--character varieties of $\pi_1(H_1)$ and $\pi_1(H_2)$ could be lifted to two smooth Lagrangian submanifolds $L_1$, $L_2$ of $\hat\clM^\frg(\Sigma')$ which are invariant under the $G$--action. They are also contained in $\mu^{-1}(0)$. Details of the definition and properties of the extended moduli space will be reviewed in Section \ref{sec_Extended}. Note that the extended moduli spaces were used in the work of Manolescu-Woodward \cite{manolescu2012floer} to define a symplectic version of instanton Floer homology.

We are then interested in studying the $G$--equivariant intersection theory of $L_1$ and $L_2$. In Section \ref{sec_Lag_int}, we introduce a transversality result on Hamiltonian perturbations that allows us to perturb $L_1$ and $L_2$ equivariantly so that their intersection is non-degenerate (see Definition \ref{def_Lag_non-deg}). Actually, in Section \ref{subsec_pert}, we introduce the notion of \emph{a compatible pair of} a holonomy perturbation $f_{\bf q}$ on the $3$--manifold $Y$ and a $G$--equivariant Hamiltonian perturbation $\Phi_{H}$ on $\hat{\clM}^{\mathfrak{g}}(\Sigma')$ such that there is a one-to-one correspondence between gauge orbits of $f_{\bf q}$--perturbed flat connections on $Y$ and $G$--orbits of $\Phi_{H}(L_1) \cap L_2$. Moreover, the non-degeneracy of connection is equivalent to the non-degeneracy of Lagrangian intersection for such a compatible pair. Therefore, we can use the equivariant transversality result on the symplectic side to construct preferred holonomy perturbation on $Y$ to simplify the calculation of generalized Casson invariants.

Let $p,q$ be two intersection points of $L_1$ and $L_2$ after perturbation, and choose a disk $D$ such that $\partial D= \gamma_1\cup \gamma_2$, where $\gamma_1$ and $\gamma_2$ are two arcs connecting $p$ and $q$ on $L_1$ and $L_2$ respectively. Suppose $H \subset G$ is a closed subgroup and for every $z \in D$ we assume that the stabilizer of $z$ contains $H$ as a subgroup. For such data, we define an \emph{equivariant Maslov index} $\mu^{H}(D)$ in Section \ref{subsec_equiv_mas_ind} which takes value in the representation ring of $H$. On the other hand, the disk $D$ defines a 1-parameter family $B_t(D)$ of (not necessarily) flat $G$ connections on $Y$ connecting the perturbed flat connections on $Y$ induced by $p$ and $q$. In \cite{bai2020equivariant}, the authors defined the \emph{equivariant spectral flow} for a family of self-adjoint operators. The main identification result of this paper is the following theorem, where the more precise statement is given in Theorem \ref{thm_main}.
 
 \begin{Theorem}
 \label{thm_index_intro}
 	The equivariant Maslov index $\mu^{H}(D)$ is equal to the equivariant spectral flow associated with the 1-parameter family of connections $B_t(D)$ on $Y$.
 \end{Theorem}

The proof of Theorem \ref{thm_index_intro} is presented in Section \ref{sec_index}. It follows from a combination of Nicolaescu's arguments in \cite{nicolaescu1995maslov}, which allow us to identify the $H$--equivariant spectral flow of $B_t(D)$ with certain infinite-dimensional Maslov index (Theorem \ref{thm_compare}), Dostoglou-Salamon's adiabatic limit result in \cite{dostoglou1994cauchy} which reduces the infinite-dimensional Maslov index to some finite-dimensional Maslov index (Theorem \ref{thm_spec_comp_finite}), and some observations regarding the relation between tangent space of the extended moduli space $\hat{\clM}^{\mathfrak{g}}(\Sigma')$ and twisted harmonic forms.

After a further identification between the symplectic areas and Chern-Simons invariants coming into the definitions of $\SU(3)$--Casson invariants on symplectic and gauge-theoretic sides, we prove

\begin{Theorem}
\label{thm_BH_intro}
Let $Y = H_1 \cup_{\Sigma} H_2$ be an integer homology $3$--sphere equipped with a Heegaard splitting such that the genus of $\Sigma$ is at least $3$. Suppose $\Phi_{H}$ is a sufficiently small $G$--equivariant Hamiltonian perturbation on $\hat{\clM}^{\mathfrak{g}}(\Sigma')$ such that $\Phi_{H}(L_1)$ intersect $L_2$ non-degenerately. Let $[\tilde{\theta}] \in \Phi_{H}(L_1) \cap L_2$ be the intersection coming from perturbing the product connection $[\theta] \in L_1 \cap L_2$. Given any $p \in \Phi_{H}(L_1) \cap L_2$, let $D(p)$ be a disc in $\hat{\clM}^{\mathfrak{g}}(\Sigma')$ as above which connects $p$ and $[\tilde{\theta}]$. For any $p$, choose a point $\hat{p}$ in $L_1 \cap L_2$ close to $p$ and let $D(\hat{p})$ be a disc connecting $\hat{p}$ and $[\theta]$ which is close to $D(p)$. Then the Boden-Herald $\SU(3)$--Casson invariant is equal to
\begin{equation}
\begin{aligned}
\lambda_{\SU(3)}(Y) &= \sum_{[p] \in (\Phi_{H}(L_1) \cap L_2)^{irr}} (-1)^{\mu(D(p))} \\
&- \sum_{[p] \in (\Phi_{H}(L_1) \cap L_2)^{red}} (-1)^{\mu_{t}(D(p))}( \mu_{n}(D(p)) - \frac{\omega(D(\hat{p}))}{2 \pi^2} +1 ).
\end{aligned}
\end{equation}
In the above formula, $(\Phi_{H}(L_1) \cap L_2)^{irr}$ is the set of $G$--orbits of $(\Phi_{H}(L_1) \cap L_2)$ whose stabilizer is isomorphic to the center of $\SU(3)$ and $(\Phi_{H}(L_1) \cap L_2)^{red}$ is the set of $G$--orbits of $(\Phi_{H}(L_1) \cap L_2)$ whose stabilizer is isomorphic to $\U(1)$ which corresponds to perturbed $\SU(2)$--flat connections on $Y$. The index $\mu(D(p))$ is the equivariant Maslov index with respect to the trivial group, while the indices $\mu_{t}(D(p))$ and $\mu_{n}(D(p))$ are the trivial and nontrivial component of the $\U(1)$--equivariant Maslov index $\mu^{\U(1)}(D(p))$ respectively. The symbol $\omega(D(\hat{p}))$ is the symplectic area of the disc $D(\hat{p})$ with respect to the symplectic form on $\hat{\clM}^{\mathfrak{g}}(\Sigma')$.
\end{Theorem}

The reader might have realized that this is parallel to the definition of Casson-Walker invariant \cite[equation (2.2)]{walker1992extension} for rational homology spheres. 

We hope that the symplectic approach used in this paper could help understand structural aspects of the generalized Casson invariants. Indeed, in \cite{walker1992extension}, Walker proved that the $\SU(2)$--Casson-Walker invariant is combinatorial in nature based on studying various isotopies of character varieties relevant to Dehn surgeries on $3$--manifolds. The combinatorial approach was generalized by Lescop in \cite{lescop1996global} to define invariants for all $3$--manifolds. The symplectic geometry provides much flexibility, though the use of Chern-Simons type invariants in our definition might put obstructions to the understanding. However, given a knot $K \subset Y$, let $Y_{1/k}(K)$ be the $3$--manifold obtained by doing Dehn surgery of $Y$ along $K$ with slope $1/k$,  it is reasonable to expect that the $\SU(n)$--Casson invariants of $Y_{1/k}(K)$ might satisfy the asymptotic growing rate $k^{n-1}$ as $k \rightarrow \infty$ for a fixed $n$. When $n = 2$, this is the classical surgery formula for Casson-Walker invariants, see \cite[Chapter 3]{walker1992extension}; when $n = 3$, such speculation is supported by the discussion in \cite[Section 6]{boden2005integer}. We will not discuss this topic further in this paper.

\subsection*{Acknowledgements} The author would like to thank his advisor John Pardon for constant support and encouragement. Special thanks are given to Boyu Zhang for numerous discussions and suggestions.

%% file: preliminary.tex

In this section, we recall some definitions and results from \cite{bai2020equivariant} which will be used in this paper. Interested readers should consult \emph{loc.cit.} for full details.

\subsection{Some linear algebra}

Let $G$ be a connected compact Lie group. Write the representation ring of $G$ as $\clR(G)$ and let $\clR(G) \otimes \bR$ be its natural $\bR$--extension.

\begin{Definition}
\label{def_clR_G}
Define $\clR_G$ to be the set of conjugation equivalence classes of the triple $(H,V,\rho)$, where $H$ is a closed subgroup of $G$, and $\rho:H\to\Hom(V,V)$ is a finite-dimensional $\bR$--linear representation of $H$. We say that $(H,V,\rho)$ is \emph{conjugation equivalent} to $(H',V',\rho')$, if there exists $g\in G$ and an isomorphism $\varphi:V\to V'$, such that
$$
	H' = g H g^{-1},
$$
and 
$$
\rho'(ghg^{-1}) =\varphi\circ \rho(h)\circ \varphi^{-1}.
$$
Given a closed subgroup $H \subset G$, let $\clR_G([H])$ be the subset of $\clR_G$ consisting of elements represented by representations of $H$. Any $\sigma \in \clR_G$ is called irreducible if it is represented by an irreducible representation.
\end{Definition}

\begin{Definition}
Let $H$ be a closed subgroup of $G$. Define
$$i^H_G: \clR_H \rightarrow \clR_G$$ be the tautological map by identifying subgroups of $H$ as subgroups of $G$.
\end{Definition}

Given a compact Lie group $H$, if $V$ is an irreducible $\bR$-linear $H$--representation, recall that $V$ is said to be of type $\bR, \bC$ or $\bH$, if $\hom_{H}(V,V) \cong \bR, \bC$ or $\bH$ respectively.

\begin{Definition}
Suppose $V$ is an orthogonal irreducible $H$--representation of type $\bK \in \{ \bR, \bC, \bH \}$. Given $r \in \bZ_{>0}$, define $d_{V}(r)$ to be the dimension of self-adjoint maps on $V^r$ intertwining with the $H$ action. In other words,
$$
d_{V}(r)=
\begin{cases}
\frac12 r(r+1) & \text{if $\bK=\bR$,} \\
r^2 & \text{if $\bK=\bC$,} \\
2r^2 - r & \text{if $\bK=\bH$}.
\end{cases}
$$
\end{Definition}

\begin{Definition}
Given $\sigma \in \clR_G$ with representative given by $(H,V,\rho)$ such that $V$ is an orthogonal $H$--representation, suppose the isotypic decomposition of $V$ as an $H$--representation is given by 
$$
	V\cong V_1^{\oplus a_1}\oplus\cdots\oplus V_m^{\oplus a_m}.
$$
Define the \emph{dimension function} of $\sigma$, denoted by $d(\sigma)$, to be the following quantity:
$$ d(\sigma) = \sum_{i=1}^{m} d_{V}(a_i). $$
\end{Definition}

This definition is independent of the choice of the representative of the equivalence class $\sigma$. According to Schur's lemma, $d(\sigma)$ is equal to the dimension of $H$--equivariant self-adjoint linear endomorphisms, denoted by $\sym_{H}(V)$, of the representation $V$. The following lemma is \cite[Lemma 3.18]{bai2020equivariant}.

\begin{Lemma}
Suppose $\sigma \in \clR_G([H])$ is represented by the triple $(H, V, \rho)$ such that $V$ is an orthogonal $H$--representation. Let $\sym_{H, \sigma}(V) =  \{ s \in \sym_{H}(V) | \sigma \cong (H, \ker(s), \rho|_{\ker(s)}) \}$. Then $\sym_{H, \sigma}(V)$ is a smooth submanifold of $\sym_{H}(V)$ of codimension $d(\sigma)$. If $K \subset \sym_{H}(V)$ is a linear space, let $\Pi: V \rightarrow \ker(s)$ be the orthogonal projection, then $s + K$ is transverse to $\sym_{H, \sigma}(V)$ at $s$ if and only if the map taking $l \in K$ to $(x \mapsto \Pi(l(x))) \in \sym_{H}(\ker(s))$ is surjective.
\end{Lemma}

\subsection{Some $3$--dimensional gauge theory}

Let $Y$ be a smooth, oriented, closed $3$--manifold equipped with a Riemannian metric. Suppose $G$ is a compact, connected and simply-connected Lie group and let $\mathfrak{g}$ be its Lie algebra. Let $P = Y \times G$ be the product principal $G$--bundle over $Y$. We use $\theta$ to represent the trivial product connection on $P$. Fix an integer $k \geq 2$. Let $\clC(Y)$ be the affine space of $L^2_k$--connections over $P$, which is an affine space modeled on $L^2_k(T^*Y \otimes \mathfrak{g})$. Let $\clG_Y$ be the $L^2_{k+1}$--gauge group of $P$. Then $\clG_Y$ acts smoothly on $\clC(Y)$ via gauge transformation. Given $B = \theta + b \in \clC(Y)$, where $b \in L^2_k(T^*Y \otimes \mathfrak{g})$, let $\Stab(B)$ and $\Orb(B)$ be the stablizer group and orbit of $B$ under the $\clG_Y$--action respectively. For $H \subset G$, let $\clC(Y)^{H}$ be the subset of $\clC(Y)$ consisting of connections whose stabilizer contains a subgroup of $G$ which is conjugate to $H$.

Recall the Chern-Simons functional $\CS$ on $\clC(Y)$ is an $\bR$--valued smooth function 
\begin{equation}
\label{eqn_def_chern_simons}
	\CS(\theta+ b):= \frac12 \int_Y d_{\theta}b \wedge b + \frac13 \int_Y [b\wedge b] \wedge b.
\end{equation} 
Using the Riemannian metric on $Y$, we have 
$$(\grad \CS)(B) = *F_B\mbox{ for all }B\in\clC(Y), $$
where $F_B = d_{\theta}b + \frac12[b ,b]$ is the curvature of $B$ and $*$ is the Hodge star operator. In particular, critical points of $\CS$ correspond to flat connections over $P$.

Next we discuss about a slight generalization of the holonomy perturbation \cite{donaldson1987orientation}, \cite{floer1988instanton}. We follow the notation of \cite{kronheimer2011knot}. This form of holonomy perturbation will be used later in Section \ref{sec_index}. Regard the circle $S^1$ as $\bR /\bZ$. We identify the open unit disc in the plane $D^2$ with the open cube $(-1, 1)^2$.

\begin{Definition}
\label{def_cyl_datum}
A \emph{cylinder datum} is a tuple $(q_1,\cdots,q_m,\beta,h_s)$ with $m\in \bZ_{> 0}$ that satisfies the following conditions.
\begin{enumerate}
\item $q_i: S^1 \times D^2 \cong S^1 \times (-1, 1)^2 \rightarrow Y$ is a smooth immersion for $i=1,\cdots,m$;
\item there exists $\epsilon > 0$ such that $q_1, \dots, q_m$ coincide on $(-\epsilon, \epsilon) \times D^2$,
\item $\beta$ is a non-negative, compactly supported smooth function on $(-1, 1)$, such that 
$$\int_{-1}^1 \beta = 1;$$
\item For any $-1 \leq s \leq 1$, $h_s :G^m\to \bR$ is a smooth function that is invariant under the diagonal action of $G$ by conjugations and the family $\{ h_s \}_{-1 \leq s \leq 1}$ depends smoothly on $s$. 
\end{enumerate}
\end{Definition}

Given $B \in \clC(Y)$ and cylinder datum ${\bf q} = (q_1,\cdots,q_m,\beta,h_s)$, we can define the \emph{holonomy map} of $B$ by
\begin{align*}
	\hol_{\bf q}(B) : D^2 &\to G^m \\
	z &\mapsto \big(\hol_{q_1,z}(B),\cdots,\hol_{q_m,z}(B)\big),
\end{align*}
where $\hol_{q_i,z}(B)$ is the element in $G$ obtained by the holonomy of $B$ along the loop $q_i(S^1 \times \{z \})$ for $1 \leq i \leq m$ using the trivialization $P_{\{0\} \times \{z\}} \cong G$. Then the \emph{cylinder function} associated to ${\bf q}$ is defined to be the map
\begin{align*}
	f_{{\bf q}}: \clC(Y) &\to \bR \\
	B &\mapsto \int_{-1}^{1} \int_{-1}^1 \beta(\lambda) h_s(\hol_{\bf q}(B)) d\lambda ds,
\end{align*}
where we use the coordinate $z = (\lambda, s)$ under the identification $D^2 \cong (-1, 1)^2$. When $h_s = \chi(s) h$ such that $h: G^m \rightarrow \bR$ is a smooth function invariant under the diagonal action of $G$ by conjugations and $\chi: (-1, 1) \rightarrow \bR$ is a compactly supported smooth function with $\int_{-1}^{1} \chi(s)ds = 1$, the $2$--form $\beta(\lambda) \chi(s) d\lambda \wedge ds$ defines a non-negative bump $2$--form over $D^2$ with total integration $1$ and $f_{{\bf q}}$ is the usual cylinder function as in \cite{kronheimer2011knot}. These cylinder functions are smooth over $\clC(Y)$ and have well-defined formal gradients. Then one can follow the procedure in \cite[Definition 3.6]{kronheimer2011knot} to construct the Banach space of holonomy perturbations, which is denoted by $\clP$. Given any $\pi \in \clP$, write the associated function over $\clC(Y)$ as $f_\pi$ and let $V_{\pi}$ be the formal gradient of $f_{\pi}$. Let $DV_{\pi}$ be the derivative of $V_{\pi}$.

\begin{Definition}
A connection $B$ is called $\pi$--flat if 
$$*F_B + V_{\pi}(B) = 0.$$
In other words, $B$ is a critical point of the functional $\CS + f_{\pi}$.
\end{Definition}

Given a pair $B \in \clC(Y)$ and $\pi \in \clP$, we can define a self-adjoint Fredholm operator with index zero which has discrete real spectrum 
$$
K_{B,\pi}: L_{k}^2(\frg)\oplus L_k^2(T^*Y\otimes \frg) \to L_{k-1}^2(\frg)\oplus L_{k-1}^2(T^*Y\otimes \frg) 
$$
by 
\begin{equation}
\label{eqn_def_K}
	K_{B,\pi}(\xi,b):=( d_B^*b, d_B\xi + *d_B b+ DV_\pi(B) (b) ).
\end{equation}
When $B$ is $\pi$--flat, define $\Hess_{B,\pi}$ to be the operator
$$
\Hess_{B,\pi}: \ker d^*_B\cap L_k^2(T^*Y\otimes \frg) \to \ker d^*_B\cap L_{k-1}^2(T^*Y\otimes \frg)
$$
given by 
$$\Hess_{B,\pi} := *d_B+DV_\pi(B).$$ 
This is a self-adjoint Fredholm operator with index zero. Note that both $K_{B, \pi}$ and $\Hess_{B,\pi}$ are $\Stab(B)$--equivariant operators.

\begin{Definition}
A $\pi$--flat connection $B$ is called \emph{non-degenerate} if $\Hess_{B,\pi}$ is an isomorphism. A perturbation $\pi \in \clP$ is called \emph{non-degenerate} if all the critical points of $\CS + f_{\pi}$ are non-degenerate.
\end{Definition}

Because the classical holonomy perturbations all arise as special cases of our formulation of holonomy perturbation, the transversality results proved in \cite[Section 4]{bai2020equivariant} all carry over due to the abundance of perturbations. The following proposition is \cite[Corollary 4.20, 4.21]{bai2020equivariant}.

\begin{Proposition}
The set $\clP^{reg}\subset \clP$ of non-degenerate holonomy perturbations is of Baire second category.

For any pair $\pi_0, \pi_1 \in \clP^{reg}$, one can find a generic smooth path $\pi_t: [0,1] \to \clP$ from $\pi_0$ to $\pi_1$, such that there are only countably many $t$ where $\pi_t$ is degenerate. Moreover, for every such $t$ there is exact one degenerate critical orbit $\Orb(B)$ of $\CS+f_{\pi_t}$, and the kernel of $\Hess_{B,\pi_t}$ is an irreducible representation of $\Stab(B)$.
\end{Proposition}

Next we recall definitions and results which lead to the definition of $\SU(n)$--Casson invariants using gauge theory.

Let $\mathcal{V}$ be a Hilbert space and $\mathcal{D} \subset \mathcal{V}$ is a space whose image in $\mathcal{V}$ is dense and the inclusion map is a compact operator. Let $D_t: \mathcal{D} \rightarrow \mathcal{V}, t \in [0, 1]$ be a $1$--parameter family of self-adjoint operators whose spectra are real and discrete. Let $H$ be a connected compact Lie group and $\mathcal{V}$ admits a linear $H$--action which preserves $\mathcal{D}$. Suppose that for all $0<t<1$, the operator $D_t$ is equivariant with respect to the $H$--action. Given any $W$ which is a finite-dimensional irreducible $H$--representation, we have an induced family of self-adjoint operators on $\Hom_{H}(W, \mathcal{V})$ and let the spectral flow of this family be $n_W$. We use the convention that if $D_0$ or $D_1$ has nontrivial kernel, the spectral flow of the family $\{D_t\}_{0\leq t \leq1}$ is calculated by the spectral flow of $\{D_t + \epsilon \cdot \id\}_{0\leq t \leq1}$ where $\epsilon > 0$ is sufficiently small. Let $ \clR^{irr}(H)$ be the set of finite-dimensional irreducible $H$--representations.

\begin{Definition}
\label{def_equi_spectral_flow}
Define the \emph{$H$--equivariant spectral flow} of $\{ D_t \}_{0 \leq t \leq 1}$, denoted by $Sf_{H}(D_t)$ to be the following element in the representation ring $\clR(H)$:
\begin{equation}
\sum_{W \in \clR^{irr}(H)} n_W \cdot [W].
\end{equation}
\end{Definition}

Let $\pi \in \clP$ be a non-degenerate perturbation and suppose $B$ is a $\pi$--flat connection. Suppose $H$ is a closed subgroup of $G$ and the group $\Stab(B) = H$.

\begin{Definition}
Define $Sf_{H}(B, \pi) \in \clR(H)$ to be the $H$--equivariant spectral flow given by the homotopy from $K_{B, \pi}$ to $K_{\theta, 0}$ constructed from the linear homotopy from $(B, \pi)$ to $(\theta, 0)$.
\end{Definition}
Because of the homotopy invariance of the spectral flow, $Sf_{H}(B, \pi)$ could be calculated by any path of $H$--equivariant operators connecting $K_{B, \pi}$ and $K_{\theta, 0}$. 

Let $n \geq 2$ be an integer. Suppose the Lie group $G$ is given by $\SU(n)$ and $\bC^n$ is the standard representation of $\SU(n)$. We further assume that $Y$ is an integer homology sphere. Let $\Sigma_{n}$ be the set of tuples of pairs of integers
$$((n_1, m_1), \dots, (n_r, m_r)), r \geq 1$$
such that $n = \sum_{i=1}^r n_i m_i$. We also assume that $n_1 \leq \cdots \leq n_r$ and $m_i$'s are in non-decreasing order if the corresponding $n_i$'s are the same. By the discussion in \cite[Section 5]{bai2020equivariant}, given any connection $B \in \clC(Y)$, there exists $\sigma \in \Sigma_{n}$, such that the vector bundle
$$E = P \times_{\SU(n)} \bC^n$$
could be decomposed into
\begin{equation}
\label{eqn_bundle_decomp}
E \cong E(n_1)^{\oplus m_1} \oplus \cdots \oplus E(n_r)^{\oplus m_r}
\end{equation}
satisfying
\begin{enumerate}
\item $E(n_i)$ is a trivial $\bC$--vector bundle over $Y$ of rank $n_i$ for $1 \leq i \leq r$;
\item The connection $B$ preserves this decomposition and is decomposed into direct sums of irreducible unitary connections on each component $E(n_i)$.
\end{enumerate}
Such a connection $B$ is called \emph{of type $\sigma$}. Fix a $\sigma \in \Sigma_{n},$ connections with above decomposition define a subset $\clC_{\sigma}(Y) \subset \clC(Y)$. All elements in $\clC_{\sigma}(Y)$ have stabilizers conjugate to a subgroup $H_{\sigma} \subset \SU(n)$. Then $\clC_{\sigma}(Y) \subset \clC(Y)^{H_{\sigma}}$ and the linear path from $B \in \clC_{\sigma}(Y)$ to $\theta$ stays in $\clC(Y)^{H_{\sigma}}$.

Suppose $g: E \rightarrow E$ is a gauge transformation that is decomposed into the direct sum of gauge transformations $g_i: E(n_i) \rightarrow E(n_i)$ using \eqref{eqn_bundle_decomp}. If $B \in \clC_{\sigma}(Y)$, one can find $\tau_1, \dots, \tau_{r} \in \clR(H_{\sigma})$ such that the $H_{\sigma}$--equivariant spectral flow from $K_{B, \pi}$ to $K_{g(B), \pi}$ is given by
$$\sum_{n_i \geq 2} \deg(g_i) \cdot \tau_i$$
which is independent of $B$ and $\pi \in \clP$, where $\deg(g_i)$ is the degree of the map $H_{3}(Y;\bZ) \rightarrow H_{3}(\U(n);\bZ)$ induced by $g_i$. If $B$ is furthermore a flat connection and let $B_i$ be the $E(n_i)$--component of $B$, define
\begin{equation}
\CS_{\sigma}(B) = \sum_{n_i \geq 2} \frac{\CS(B_i)}{4 \pi^2 n_i} \cdot \tau_i \in \clR(H_{\sigma}) \otimes \bR.
\end{equation}

Now we are ready to recall the definition of the $\SU(n)$--Casson invariant. Let $\clC_{F}(Y) \subset \clC(Y)$ be the subspace consisting of flat connection on $P$. Let $\mathcal{U} \subset \clC(Y)$ be an open subset containing $\clC_{F}(Y)$ such that the inclusion
$$\clC_{F}(Y) \cap \clC(Y)^{H_{\sigma}} \hookrightarrow \mathcal{U} \cap \clC(Y)^{H_{\sigma}}$$
induces a one-to-one correspondence on connected components. By Uhlenbeck compactness, there exists $r_0 > 0$ such that as long as $\pi \in \clP$ satisfies $\| \pi \| < r_0$, the space of $\pi$--flat connections is contained in $U$. 

\begin{Definition}
\label{defn_ind}
Suppose $\| \pi \| < r_0$ is a non-degenerate perturbation and let $B$ be a $\pi$--perturbed flat connection of type $\sigma \in \Sigma_{n}$. Take $g \in \clG_{Y}$ such that $g(B) \in \clC(Y)^{H_{\sigma}}$ and let $g(\hat{B})$ be a genuine flat connection lying in the same connected component of $\mathcal{U} \cap \clC(Y)^{H_{\sigma}}$ as $g(B)$. Define $\ind(B, \pi)$ to be
\begin{equation}
Sf_{H_{\sigma}}(g(B), \pi) - [\ker d_{g(B)}] - \CS_{\sigma}(g(\hat{B})) \in \clR(H_{\sigma}) \otimes \bR,
\end{equation}
where $[\ker d_{g(B)}]$ is the $H_{\sigma}$--representation associated to the Lie algebra of $\Stab(B)$.
\end{Definition}

The following is proved in \cite[Section 5]{bai2020equivariant}.

\begin{Theorem}
There exists a set $\widetilde{\clR}_{\SU(n)}$, a well-defined map $\clR(H_{\sigma}) \otimes \bR \rightarrow \widetilde{\clR}_{\SU(n)}$ and a $\bZ$--submodule $\widetilde{\Bif}_{\SU(n)} \subset \bZ \widetilde{\clR}_{\SU(n)}$ with the following significance. Let $[\ind(B, \pi)]$ be the image of $\ind(B, \pi)$ in $\widetilde{\clR}_{\SU(n)}$. If $\| \pi \| < r_0$ is a non-degenerate perturbation, define
$$\ind_{\pi} := \sum_{\Orb(B) \mbox{ is $\pi$-- flat}} [\ind(B, \pi)] \in \bZ \widetilde{\clR}_{\SU(n)}.$$
Then for any pair of non-degenerate $\pi_{1}, \pi_{2} \in \clP$ such that $\| \pi_{i} \| < r_0$ for $i=1,2$, we have
$$\ind_{\pi_1} - \ind_{\pi_2} \in \widetilde{\Bif}_{\SU(n)}.$$
\end{Theorem}

Therefore the image of $\ind_{\pi}$ in the $\bZ$--module $\bZ \widetilde{\clR}_{\SU(n)} / \widetilde{\Bif}_{\SU(n)}$ is a topological invariant of the integer homology sphere $Y$. One can then use this quantity to construct Casson-type invariants using different normalizations. In particular, the $\bR$--valued $\SU(3)$--Casson invariant of Boden-Herald \cite{boden1998the} could be recovered as a special case of the above construction.

Our goal is to present $\ind(B, \pi)$ using geometric quantities constructed from finite-dimensional symplectic geometry when $Y$ is equipped with a Heegaard splitting and the perturbation $\pi$ has certain adapted behavior. This is what the \emph{symplectic formula} is referring to.

%% file: Extended_moduli.tex

In this section, we recall the construction of the extended moduli space in \cite{jeffrey1994extended} and discuss about some of its properties. Most importantly, we will see that the essential information of the tangent spaces of the extended moduli space is encoded in the corresponding space of twisted harmonic forms. We also recall the construction of Lagrangian submanifolds in extended moduli spaces from \cite{manolescu2012floer} which will be used in this paper. 

Suppose $\Sigma$ is a closed Riemann surface with genus $h \geq 2$ equipped with a Riemannian metric. Let $p \in \Sigma$ be a base point. Define $\Sigma' = \Sigma \backslash D^2_p$ to be the Riemann surface with one boundary component obtained by removing a disc $D^2_p$ around $p$. Suppose $G$ is a simply-connected compact Lie group and $\mathfrak{g}$ is its Lie algebra. We identify $\mathfrak{g}$ with its dual $\mathfrak{g}^*$ using the Killing form. Let $P$ be the product principal $G$--bundles over $\Sigma$. We continue to use $\theta$ to represent the trivial product connection.

Fix an integer $k \geq 2. $Let $\clA(\Sigma)$ and $\clA(\Sigma')$ be the affine space of $L_k^2$--connections over $P$ and $P|_{\Sigma'}$ respectively. Denote the $L^2_{k+1}$--gauge groups of $P$ and $P|_{\Sigma'}$ by $\clG_{\Sigma}$ and $\clG_{\Sigma'}$ respectively. Let $\mathcal{M}(\Sigma) = \clA_F(\Sigma) / \clG_{\Sigma}$ be the moduli space of flat connections over $P$, where $\clA_F(\Sigma)$ is the space of flat connections over $P$. We also need the group $\clG_{\Sigma'}^c$ consisting of elements in $\clG_{\Sigma'}$ which are identity near the boundary of $\Sigma'$.

\begin{Definition}
Let $s$ be the coordinate on $S^1 \cong \partial \Sigma'$. Define
$$\clA_{F}^{\mathfrak{g}}(\Sigma') := \{ A \in \clA(\Sigma') \big| F_A = 0,  A = \theta + \xi ds \mbox{ near the boundary for some $\xi \in \mathfrak{g}$} \}.$$
Then the \emph{extended moduli space} $\mathcal{M}^{\mathfrak{g}}(\Sigma')$ is defined to be the quotient
$$\mathcal{M}^{\mathfrak{g}}(\Sigma') := \clA_{F}^{\mathfrak{g}}(\Sigma') / \clG_{\Sigma'}^c.$$
\end{Definition}

$\mathcal{M}^{\mathfrak{g}}(\Sigma')$ has a finite-dimensional description similar to character varieties. Let us choose $\{ \alpha_i, \beta_i \}_{1 \leq i \leq h}$ to be a set of simple loops on $\Sigma$ with common base point at some $p' \in \partial \Sigma' \subset \Sigma' \subset \Sigma$ which gives the standard presentation of $\pi_{1}(\Sigma)$. Choose a simple loop $\gamma$ in $\Sigma'$ based at $p'$ wrapping around the boundary once then we have $\Pi_{i=1}^h [\alpha_i, \beta_i] = \gamma$. Using the holonomy map, one can show that
\begin{equation}
\label{eqn_extended}
\mathcal{M}^{\mathfrak{g}}(\Sigma') = \{A_i, B_i \in G, \forall 1\leq i \leq h, \xi \in \mathfrak{g} \big| \Pi_{i=1}^h [A_i, B_i]  =\exp(\xi) \}.
\end{equation}
We will use both the gauge-theoretic definition and the representation-theoretic definition in the sequel. Note that $\mathcal{M}^{\mathfrak{g}}(\Sigma')$ admits a $G$--action: given $g \in G$, it acts on $\mathcal{M}^{\mathfrak{g}}(\Sigma')$ by
$$((A_i, B_i)_{1 \leq i \leq h}, \xi) \mapsto ((g A_i g^{-1}, g B_i g^{-1})_{1 \leq i \leq h}, \Ad_{g}(\xi)),$$
where $\Ad_{g}$ is the adjoint action of $g$ on $\mathfrak{g}$.
Use the Killing form on $\mathfrak{g}$ to define a norm on $\mathfrak{g}$, the next proposition is proved in \cite{jeffrey1994extended}. 

\begin{Proposition}
\label{prop_extended}
\begin{enumerate}
\item The open $G$--invariant subset $\hat{\mathcal{M}}^{\mathfrak{g}}(\Sigma') \subset \mathcal{M}^{\mathfrak{g}}(\Sigma')$ defined by $|\xi| < \delta$ for some sufficiently small $\delta > 0$ is actually a smooth manifold;
\item There exists a symplectic form $\omega$ on $\hat{\mathcal{M}}^{\mathfrak{g}}(\Sigma')$ and the $G$--action defines a smooth Hamiltonian action with respect to $\omega$. The moment map $\mu: \hat{\mathcal{M}}^{\mathfrak{g}}(\Sigma') \rightarrow \mathfrak{g}^*$ of this $G$--action is given by 
$$((A_i, B_i)_{1 \leq i \leq h}, \xi) \mapsto \xi;$$
\item The moment map reduction $\mu^{-1}(0) / G$ is naturally identified with $\mathcal{M}(\Sigma)$;
\item If $[A] \in \mathcal{M}(\Sigma)$ is represented by some $A \in \clA_{F}^{\mathfrak{g}}(\Sigma')$, then the stabilizer group of $A$ under the $\clG_{\Sigma'}$--action, written as $\Stab(A)$, is isomorphic to $\Stab([A])$, the stabilizer of $[A] \in \mathcal{M}(\Sigma)$ under the $G$--action.
\end{enumerate}
\end{Proposition}

Making $\delta$ smaller if necessary, we can assume that the exponential map $\exp: \mathfrak{g} \rightarrow G$ is a diffeomorphism onto its image over $\{ \xi \in \mathfrak{g} \big| | \xi | < \delta \}$. Then the following statement follows from the equation \eqref{eqn_extended}.

\begin{Lemma}
\label{lem_trivial}
The map $\hat{\mathcal{M}}^{\mathfrak{g}}(\Sigma') \rightarrow G^{2h}$ defined by mapping $((A_i, B_i)_{1 \leq i \leq h}, \xi)$ to the $G^{2h}$--component is a $G$--equivariant embedding, where $G^{2h}$ is equipped with the diagonal conjugation action. \qed
\end{Lemma}

Introduce the notation
$$\clA^0_{F}(\Sigma) := \{ A \in \clA_{F}(\Sigma) \big| A|_{D^2_p} \mbox{ is given by the product connection} \}.$$
Then $\clA^0_{F}(\Sigma)$ could be viewed as a subspace of $\clA_{F}^{\mathfrak{g}}(\Sigma')$ corresponding to $\xi = 0$. Choose $A \in \clA^0_{F}(\Sigma)$. Denote by $[A] \in \hat{\mathcal{M}}^{\mathfrak{g}}(\Sigma')$ the image of $A$ in $\hat{\mathcal{M}}^{\mathfrak{g}}(\Sigma')$, we now give an explicit description of the tangent space $T_{[A]} \hat{\mathcal{M}}^{\mathfrak{g}}(\Sigma')$. 

Let $\Omega^{i}(\Sigma') \otimes \mathfrak{g}$ (resp. $\Omega^{i}(\Sigma) \otimes \mathfrak{g}$) be the space of $\mathfrak{g}$--valued $i$--forms on $\Sigma'$ (resp. $\Sigma$) for $i=0,1,2$. Also write the subspace of $\Omega^{i}(\Sigma') \otimes \mathfrak{g}$ consisting of elements with compact support as $\Omega^{i}_c(\Sigma') \otimes \mathfrak{g}$. Recall that $s$ is the coordinate on $\partial \Sigma'$. Define
$$\Omega^{1, \mathfrak{g}}(\Sigma') = \{ a \in \Omega^{1}(\Sigma') \otimes \mathfrak{g} \big| a = \xi ds \mbox{ near the boundary for some $\xi \in \mathfrak{g}$ } \}.$$
We have the elliptic complex
\begin{equation}
\label{eqn_def_cpx}
\begin{tikzcd}
\Omega_c^0(\Sigma') \otimes \mathfrak{g} \arrow[r, "d_A"] & {\Omega^{1,\mathfrak{g}}(\Sigma')} \arrow[r, "d_A"] & \Omega_c^2(\Sigma') \otimes \mathfrak{g}.
\end{tikzcd}
\end{equation}
Then the tangent space $T_{[A]} \hat{\mathcal{M}}^{\mathfrak{g}}(\Sigma')$ is identified with the first cohomology group of \eqref{eqn_def_cpx}, written as $\tilde{H}_{A}^{1, \mathfrak{g}}$. The bilinear form 
$$\omega_{\Sigma'}(a_1, a_2) = \int_{\Sigma'} \Tr(a_1 \wedge a_2)$$ 
over $\Omega^{1,\mathfrak{g}}(\Sigma')$ defines the symplectic form on $\hat{\mathcal{M}}^{\mathfrak{g}}(\Sigma')$ under this identification.

On the other hand, view $A \in \clA_{F}(\Sigma)$, we have the complex 
\begin{equation}
\begin{tikzcd}
\Omega^0(\Sigma) \otimes \mathfrak{g} \arrow[r, "d_A"] & {\Omega^1(\Sigma) \otimes \mathfrak{g}} \arrow[r, "d_A"] & \Omega^2(\Sigma) \otimes \mathfrak{g}.
\end{tikzcd}
\end{equation}
The cohomolgy groups $H^{i}_{A}$ for $i=0,1,2$ could be identified with the space of twisted harmonic $i$--forms. Similarly, the bilinear form 
$$\omega_{\Sigma}(a_1, a_2) = \int_{\Sigma} \Tr(a_1 \wedge a_2)$$ 
over $\Omega^1(\Sigma) \otimes \mathfrak{g}$ is non-degenerate and descends down to a non-degenerate pairing on $H^{1}_{A}$.

We need to introduce some auxiliary objects to understand the relation between $\tilde{H}_{A}^{1, \mathfrak{g}}$ and $H^{1}_{A}$ as symplectic vector spaces. Let $\overline{\Sigma} \cong \Sigma \backslash \{p\}$ be the Riemann surface obtained from $\Sigma'$ by attaching a cylindrical end $S^1 \times [0, +\infty)$ along $\partial \Sigma'$. We still use $P$ to represent the product principal $G$--bundle over $\overline{\Sigma}$. Choose $\epsilon > 0$ to be a positive real number. Define $L^{2, \epsilon}_{k}(\Omega^{i}(\overline{\Sigma}) \otimes \mathfrak{g})$ the completion of the space of $\mathfrak{g}$--valued $i$--forms over $\Sigma'$ under the norm
$$\| a \|_{L^{2, \epsilon}_{k}} = (\int_{\overline{\Sigma}} e^{\epsilon \chi} (|a|^2 + |\nabla a|^2 + \cdots + |\nabla^k a|^2))^{\frac12},$$
where $\chi$ is a smooth function on $\overline{\Sigma}$ which is equal to $1$ over $S^1 \times [0, +\infty)$ and $\nabla$ is some background connection. This is the norm associated to the inner product
$$\langle a_1, a_2 \rangle_{k, \epsilon} = \int_{\overline{\Sigma}} e^{\epsilon \chi}( \langle a_1, a_2 \rangle + \langle \nabla a_1, \nabla a_2 \rangle + \cdots + \langle \nabla^k a_1, \nabla^k a_2 \rangle).$$
Note that the bilinear form on $L^{2, \epsilon}_{k}(\Omega^{i}(\overline{\Sigma}) \otimes \mathfrak{g})$ given by
$$\omega_{\overline{\Sigma}}(a_1, a_2) = \int_{\overline{\Sigma}} \Tr(a_1 \wedge a_2)$$
is non-degenerate, which could be proved using the Hodge $*$--operator.
Given $A \in \clA^0_{F}(\Sigma)$, it naturally defines a connection on $P|_{\overline{\Sigma}}$, namely extending $A$ by the product connection on the cylindrical end. Therefore, we can define the operator 
$$d_A : L^{2, \epsilon}_{k}(\Omega^{i}(\overline{\Sigma}) \otimes \mathfrak{g}) \rightarrow L^{2, \epsilon}_{k-1}(\Omega^{i+1}(\overline{\Sigma}) \otimes \mathfrak{g})$$
using the connection $A$ the the $L^2$--formal adjoint operator of $d_A$
$$d_A^* : L^{2, \epsilon}_{k}(\Omega^{i}(\overline{\Sigma}) \otimes \mathfrak{g}) \rightarrow L^{2, \epsilon}_{k-1}(\Omega^{i-1}(\overline{\Sigma}) \otimes \mathfrak{g}).$$

\begin{Lemma}
For $\epsilon > 0$ sufficiently small, the kernel of the map
$$d^*_A \oplus d_A : L^{2, \epsilon}_{k}(\Omega^{1}(\overline{\Sigma}) \otimes \mathfrak{g}) \rightarrow L^{2, \epsilon}_{k-1}(\Omega^{0}(\overline{\Sigma}) \otimes \mathfrak{g}) \oplus L^{2, \epsilon}_{k-1}(\Omega^{2}(\overline{\Sigma}) \otimes \mathfrak{g})$$
is naturally isomorphic to $H_A^1$ and this isomorphism respects the symplectic structure.
\end{Lemma}
\begin{proof}
The map $\ker(d_A \oplus d_A^*) \rightarrow H_A^1$ is given by the natural inclusion $L^{2, \epsilon}_{k}(\Omega^{1}(\overline{\Sigma}) \otimes \mathfrak{g}) \rightarrow L^{2}_{k}(\Omega^1(\Sigma))$, after identifying $\overline{\Sigma}$ with $\Sigma \backslash \{p\}$. This map is the same as the composition $H^{1}_{A}(\Sigma', \partial \Sigma') \cong H^{1}_{c}(\Sigma', d_A)  \rightarrow H^1_A$, which is an isomorphism by counting dimensions. This map obviously respects the symplectic structure.
\end{proof}

\begin{Lemma}
There exists a map $P_0: \ker(d_A \oplus d_A^*) \rightarrow \tilde{H}_A^{1, \mathfrak{g}}$ which defines a linear symplectic embedding. The symplectic orthogonal complement of $\ima(P_0)$ is isomorphic to $(\mathfrak{g}/ H_A^0) \oplus (\mathfrak{g}/ H_A^0)^*$, where $H_A^0$ is identified with the Lie algebra of the stabilizer group of $A \in \clA_{\Sigma}$. Moreover, the induced symplectic structure on $(\mathfrak{g}/ H_A^0) \oplus (\mathfrak{g}/ H_A^0)^*$ from $\omega_{\Sigma'}$ is the same as the canonical linear symplectic form on $T^*(\mathfrak{g}/ H_A^0) \cong (\mathfrak{g}/ H_A^0) \oplus (\mathfrak{g}/ H_A^0)^*$.
\end{Lemma}
\begin{proof}
The proof is essentially the same as the proof of \cite[Proposition 3.9]{jeffrey1994extended}. In \emph{loc.cit}, Jeffrey constructed a family of linear symplectic embeddings $P_r :\ker(d_A \oplus d_A^*) \rightarrow \tilde{H}_A^{1, \mathfrak{g}}$ parametrized by $r \in [0, +\infty)$ and the $P_0$ here is the specialization at $r=0$. Although \cite[Proposition 3.9]{jeffrey1994extended} was stated for $G = \SU(2)$, the same arguments work for general Lie groups. 
\end{proof}

Using these two lemmas, and note that all the isomorphisms commute with the $\Stab(A)$--action, we conclude
\begin{Corollary}
\label{cor_slice}
Suppose $[A] \in \hat{\mathcal{M}}^{\mathfrak{g}}(\Sigma')$ is represented by some $A \in \clA_F^0(\Sigma)$, then there is a $\Stab(A)$--equivariant symplectic isomorphism
\begin{equation}
\label{eqn_harmonic}
T_{[[A]]} \hat{\mathcal{M}}^{\mathfrak{g}}(\Sigma') \cong H_A^1 \oplus T^*(\mathfrak{g}/ H_A^0),
\end{equation} 
where $T^*(\mathfrak{g}/ H_A^0)$ is equipped with the canonical symplectic form and trivial $\Stab(A)$--action. \qed
\end{Corollary}

\begin{remark}
The above corollary could be viewed as a consequence of the symplectic slice theorem which appeared in the proof of \cite[Theorem 2.1]{sjamaar1991stratified}. However, we cannot apply the slice theorem directly because we need to identify the tangent bundle of the normal slice explicitly as twisted harmonic $1$--forms. This will be important in our calculations of spectral flows later.
\end{remark}

Now suppose $Y'$ is a smooth, oriented, $3$--manifold with boundary $\partial Y' = \Sigma$. Choose the base point $p \in \Sigma$ as above then $p$ also define base points in both $Y'$. Assume that $\pi_1(\Sigma) \rightarrow \pi_1(Y')$ induced by the inclusion defines a surjective map. Denote by $P$ the trivial $G$--bundle on both $Y'$ and $\Sigma$. Let $\clA_{F}^{\mathfrak{g}}(\Sigma'|Y') \subset \clA_{F}^{\mathfrak{g}}(\Sigma')$ be the space consisting of elements in $\clA_{F}^{\mathfrak{g}}(\Sigma')$ which extend to flat connection on $P$ over $Y'$. Such space is preserved by the action of $\clG_{\Sigma'}^c$. Furthermore, let $\clA_{F}(Y')$ be the space of flat connections on $P|_{Y'}$ and let $\clG_0(Y')$ be the based gauge group consisting of gauge transformations $g \in \clG_{Y'}$ such that $g(p) = \id$. Then the map
$$\clA_{F}(Y') / \clG_0(Y') \rightarrow \clA_{F}^{\mathfrak{g}}(\Sigma'|Y') / \clG_{\Sigma'}^c$$
is a diffeomorphism. By \cite[Lemma 5.1]{manolescu2012floer}, the embedding
$$\clA_{F}(Y') / \clG_0(Y') \cong \clA_{F}^{\mathfrak{g}}(\Sigma'|Y') / \clG_{\Sigma'}^c \hookrightarrow\clA_{F}^{\mathfrak{g}}(\Sigma') / \clG_{\Sigma'}^c = \mathcal{M}^{\mathfrak{g}}(\Sigma')$$
defines a Lagrangian submanifold  $L_{Y'}$ in $\mathcal{M}^{\mathfrak{g}}(\Sigma')$ with respect to the symplectic form in Proposition \ref{prop_extended}. The manifold $L_{Y'}$ diffeomorphic to $G^{h}$ and it is preserved by the $G$--action on $\mathcal{M}^{\mathfrak{g}}(\Sigma')$. Additionally, this Lagrangian is contained in the $0$--level set of the moment map $\mu^{-1}(0)$. As a result, $L_{Y'}$ is contained in the smooth locus $\hat{\mathcal{M}}^{\mathfrak{g}}(\Sigma')$.

One could obtain many $G$--invariant Lagrangians in $\hat{\mathcal{M}}^{\mathfrak{g}}(\Sigma')$ by the construction above. We will study the equivariant intersection theory of a pair of such Lagrangians and define some version of equivariant intersection number which calculates the generalized Casson invariants. 

%% file: Lagrangian_intersection.tex

This section develops an equivariant transversality result for Lagrangian submanifolds inside symplectic manifolds equipped with Lie group actions, and defines the notion \emph{equivariant Maslov index} which will be used in the calculation of the generalized Casson invariants.

Suppose $(M, \omega)$ is a (not necessarily closed) smooth symplectic manifold of dimension $2n$, where $\omega$ is the symplectic form . Let $G$ be a connected, compact Lie group that acts on $(M,\omega)$ such that $\forall g \in G$, $g^{*} \omega = \omega$. Let $L_1, L_2$ be two oriented Lagrangian submanifolds in $M$ which are invariant under the $G$--action. 

For $p\in M$, we use $\Orb(p)$ to denote the orbit of $p$ under the action of $G$. Note that $\Orb(p)$ could be identified with $G / \Stab(p)$, where $\Stab(p)$ is the stabilizer group of $p$ under the $G$--action. In particular, $\Orb(p)$ is a smooth manifold. A subset $X$ of $M$ is called \emph{$G$--invariant} if $G(X)=X$; and similarly, a subbundle $E$ of $TM|_X$ where $X$ is $G$--invariant is called \emph{$G$--invariant} if $G(E)=E$.

\begin{Definition}
\label{def_Lag_non-deg}
	We say that $L_1$ and $L_2$ intersect \emph{non-degenerately} at p, if for $p\in L_1\cap L_2$, we have $T_pL_1\cap T_pL_2=T_p\Orb(p)$. $L_1$ and $L_2$ are said to intersect \emph{non-degenerately} if they intersect \emph{non-degenerately} at each of their intersection point.
\end{Definition}

\begin{remark}
Using the terminology in \cite{pozniak1999floer}, the above definition is equivalent to saying that $L_1$ and $L_2$ \emph{have clean intersection along $\Orb(p)$ for each $p \in L_1 \cap L_2$}. Note that if $L_1$ and $L_2$ are compact and intersect non-degenerately, they only intersect along finitely many $G$--orbits.
\end{remark}

\subsection{Transversality}

Suppose $(M, \omega), L_1, L_2$ are as above, we prove that $L_1$ and $L_2$ intersect non-degenerately after applying a generic $G$--equivariant Hamiltonian perturbation on $L_1$. We also characterize the failure of equivariant transversality when varying the Hamiltonian perturbations in generic $1$--paramter families.

We start by fixing some notations from symplectic geometry. Let $C^{\infty}_{G}(M)$ be the space of compactly supported smooth $G$--invariant time-independent Hamiltonians on the symplectic manifold $M$ equipped with the $C^{\infty}$--topology.

\begin{Definition}
Let $C_0^\infty \Big([0,1], C_G^\infty(M)\Big)$ be the space of time-dependent smooth $G$--invariant Hamiltonians which vanish near $s=0$ and $s=1$, where $s$ is the coordinate on $[0,1]$.
Take $\{H_{s}^{0}, H_{s}^{1}, \cdots \}$ which form a countable dense subset of $C_0^\infty \Big([0,1], C_G^\infty(M)\Big)$ and let 
$$N_m := \sup \{\|H_{s}^{0}\|_{C^m},\cdots,\|H_{s}^{m}\|_{C^m}\}.$$ 
Define $\clH$ to be the completion of the subspace of $C_0^\infty \Big([0,1], C_G^\infty(M)\Big)$ consisting of functions of form
$$\sum_{m\ge 0} a_m H_{s}^{m}$$
such that $\sum_{m\ge 0} N_m|a_m|<+\infty$ under the norm $\| \sum_{m\ge 0} a_m H_{s}^{m} \| := \sum_{m\ge 0} N_m|a_m|$.
\end{Definition}

Given any $H_{s} \in C_0^\infty \Big([0,1], C_G^\infty(M)\Big)$, the Hamiltonian vector field $X_{H_{s}}$ of $H_{s}$ is determined by
$$dH_{s} = \omega(\cdot, X_{H_{s}}).$$
Let $\Phi_{H_{s}}$ be the time $1$--flow generated by integrating $X_{H_{s}}$. Then for any $G$--invariant Lagrangian $L \subset M$, $\Phi_{H_{s}}(L)$ is also a $G$--invariant Lagrangian.

Recall that for every manifold $S$, there is a canonical symplectic form $d \lambda$ on $T^*S$ (see, for example, \cite[p. 90]{mcduff2017introduction}). When $S$ admits a smooth $G$--action, $d \lambda$ is preserved by the induced $G$--action on $T^*S$.
We need the following equivariant local version of the Lagrangian neighborhood theorem:

\begin{Lemma}
\label{lem_Lag_nbhd}
Suppose $L$ is a $G$--invariant Lagrangian submanifold of $(M, \omega)$, and suppose $\Orb(p) \subset L$.
Let $E \subset TM|_{\Orb(p)}$ be a $G$--invariant Lagrangian subbundle that is transverse to $TL|_{\Orb(p)}$. Then there exists a $G$--invariant open subset $U\subset M$ with the following property: if we write $S=U\cap L$, then there is a $G$--equivariant symplectomorphism $\varphi$ from an open neighborhood of the zero section of $T^*S$ to $U$, such that 
\begin{enumerate}
	\item $\varphi$ equals the identity map on $S$,
	\item the tangent map of $\varphi$ on $S$ sends the fibers of $T^*S$ to the fibers of $E$ on $\Orb(p)$.
\end{enumerate}
\end{Lemma} 

\begin{proof}
Because $E$ and $TL|_{\Orb(p)}$ are transverse to each other fiberwisely as Lagrangian subbundles of $TM|_{\Orb(p)}$, the map $\beta_{E}: E \rightarrow T^{*}L|_{\Orb(p)}$ given by $v \mapsto \omega(v, \cdot)$ is an isomorphism. Using the identification
$$TM|_{\Orb(p)} \cong E \oplus TL|_{\Orb(p)} \cong T^{*}L|_{\Orb(p)} \oplus TL|_{\Orb(p)},$$
we can construct a fiberwise $G$--invariant almost complex structure $J_p$ on $TM|_{\Orb(p)}$ which is compatible with the fiberwise symplectic form on $TM|_{\Orb(p)}$ such that $J_p(TL|_{\Orb(p)}) = E$. 

Extend $J_p$ to a $G$--equivariant compatible almost complex structure $J$ on $M$ and let $g_{J} = \omega(\cdot, J\cdot)$ be the compatible Riemannian metric.  After choosing sufficiently small $G$--invariant neighborhood of $\Orb(p) \in T^{*}L$ so that exponential map of $g_{J}$ is well defined, the argument of \cite[Theorem 3.4.13]{mcduff2017introduction} carries over. Our special choice of the almost complex structure guarantees the second condition holds.
\end{proof}

\begin{Lemma}
\label{lem_exact}
Let $S$ be a $G$--manifold and suppose that $U$ is a sufficiently small $G$--invariant neighborhood of $\Orb(p) \subset S$. Then any $G$--invariant Lagrangian $L \subset T^{*} S$ which is transverse to the Lagrangian foliation on $T^{*} U$ induced by cotangent fibers is locally presented as the graph of an exact $G$-equivariant $1$--form over $U$.
\end{Lemma}

\begin{proof}
Because $L$ is transverse to the Lagrangian foliation on $T^{*} U$ induced by cotangent fibers, it is locally represented by the graph of some $1$--form $\theta_L$ over $U$. The Lagrangian condition implies that $\theta_L$ is closed. 

Using the slice theorem, without loss of generality we can assume that $U = G \times_{\Stab(p)} U^{\prime}$ where $U^{\prime}$ is a contractile $G$--invariant neighborhood of the origin in $T_p S / T_p \Orb(p)$. Because the quotient of $U$ by $G$ is contractible, the cohomology class represented by the equivariant form $\theta_L$ is 0. Using the de Rham complex of $G$--equivariant differential forms on $U$, we conclude that $\theta_L$ is exact.
\end{proof}

Now we are ready to establish the equivariant transversality result. Let $A\subset M$ be a compact subset such that $L_1\cap L_2$ is included in the interior of $A$.  Let $\clH_0\subset \clH$ be a sufficiently small open ball centered at zero, such that for every $H_{s} \in \clH_0$, the map $\Phi_{H_{s}}$ is defined on $A$, and $\Phi_{H_{s}}(L_1)\cap L_2$ is included in the interior of $A$.

Suppose $\sigma\in \clR_G$. Define $\mathcal{M}_\sigma\subset \clH_0\times (A\cap L_2)$ to be the set of points $(H_{s},p)$ such that 
\begin{enumerate}
	\item $p\in \Phi_{H_{s}}(L_1)\cap L_2$,
	\item The linear space $\big(T_p  \Phi_{H_{s}}(L_1) \cap T_p(L_2)\big) / T_p\Orb(p)$ as a $\Stab(p)$--representation represents $\sigma$.
\end{enumerate}

\begin{Proposition}
\label{prop_universal_moduli_space}
	The space $\clM_\sigma$ is a Banach manifold. The projection of $M_\sigma$ to $\clH_0$, written as $\pi_{\sigma}$, is Fredholm with index $-d(\sigma)$. 
\end{Proposition}

\begin{proof}
	Let $(H_{s},p)\in \clM_\sigma$. Let $Q_p\subset TM|_{\Orb(p)}$ be a $G$--invariant Lagrangian subbundle such that $Q_p$ is transverse to $T\Phi_{H_{s}}(L_1)|_{\Orb(p)}$ and $T L_2|_{\Orb(p)}$ along $\Orb(p)$. Let $B_2$ be a small $G$--invariant open neighborhood of $p$ in $L_2$. Then by Lemma \ref{lem_Lag_nbhd}, there exists be a $G$--invariant open neighborhood $B$ of $p$ in $M$, such that 
	\begin{enumerate}
		\item There exists a $G$--equivariant symplectomorphism to a neighborhood of the zero section of $T^*B_2$,
		\item $T^*_p B_2$ is tangent to $Q_p$ at $p$ under the above diffeomorphism.
	\end{enumerate}
	After further shrinking $B_2$ and $B$ if necessary, by Lemma \ref{lem_exact}, $\Phi_{H_{s}}(L_1)\cap B$ is the graph of the differential of a $G$--invariant function on $\overline{B_2}$. This function is uniquely determined if we require that its integration over $\overline{B_2}$ is $0$. Therefore we have a map from a sufficiently small neighborhood $N(H_{s})$ of $H_{s}$ to $C^{\infty}_{G}(\overline{B_2})$
	$$
	\pi: N(H_{s}) \to C_G^\infty(\overline{B_2})
	$$
	such that for each $H_{s}^{\prime}\in N(H)$, we have $\Phi_{H_{s}^{\prime}}(L_1)\cap \overline{B}$ is the graph of $d\pi(H_{s}^{\prime})$, and $\int_{\overline{B_2}}\pi(H_{s}^{\prime})=0$. 
\begin{claim}
 The map $\pi$ is a submersion at $H_{s}$ onto the subspace
	$$
	\big\{g\in C_G^\infty(\overline{B_2}) | \int_{\overline{B_2}} g = 0\big\}.
	$$ 
\end{claim}
	
	For each $g\in C_G^\infty(\overline{B_2})$ such that $\int_{\overline{B_2}} g =0$, let $\hat g\in C_G^\infty(M)$ be a smooth, compactly supported function on $M$, such that $\hat g|_{B}$ is given by the pull-back of $g$ from $B_2$ to $T^*B_2$. Let $\chi:[0,1]\to \bR^{\ge 0}$ be a smooth, compactly supported function such that $\int \chi =1$. For $s\in[0,1]$, let $\Phi_{H_{s}}(s)$ be the integration of $X_{H_{s}}$ from $s$ to $1$. Define
	$$
	\hat{H_s} = \chi(s) \cdot\big(\hat g\circ  \Phi_{H_{s}}(s)\big).
	$$
	Then 
	$$
	\frac{d}{dv} \pi(H+v\hat H) = g.
	$$
	Therefore the claim is proved.

Let $D$ be a slice of $B_2$ under the $G$--action through $p$. Let $D^0$ be the fixed point subset of $D$ under the $\Stab(p)$--action. Then for any $(\Phi_{H^{\prime}_s}, q) \in \clM_\sigma$ such that $H^{\prime}_s \in N(H_{s})$ and $q \in B$, we have $q \in D^0$ where $D^0$ is viewed as a subset of the zero section of the Weinstein neighborhood. Choose a $G$--invariant Riemannian metric on $B_2$ and let $S_p$ be the orthogonal complement of $T_{p} \Orb(p)$ in $T_p B_2$ and let $S_p^0$ be the fixed point subset of $S_p$ under the $\Stab(p)$--action. Using parallel transports, we can identify $TD$ with $S_p$. Now we can define a map
$$\psi: N(H_{s}) \times D^0 \rightarrow (S_p^0)^* \times \sym_{\Stab(p)}(S_p)$$
$$(H_{s}^{\prime}, q) \mapsto (\nabla_{q}(\pi(H_{s}^{\prime}))|_{D^0}, \Hess(\pi(H_{s}^{\prime}))|_{S_p}),$$
where $\nabla$ and $\Hess$ stand for the gradient and Hessian respectively. It is easy to see that 
$$\clM_{\sigma} = \psi^{-1}(\{ 0 \} \times \sym_{\Stab(p), \sigma}(S_p)).$$
Using the abundance result \cite[Lemma 2.25]{bai2020equivariant} and the result established in the claim, we conclude that the differential of $\psi$ at $(H_{s},p)$ is surjective. Therefore $\clM_{\sigma}$ is a Banach manifold near $(H_{s},p)$. Because the inclusion 
$$\psi^{-1}(\{ 0 \} \times \sym_{\Stab(p), \sigma}(S_p)) \hookrightarrow N(H_{s}) \times D^0$$ is a Fredholm map of index $- \dim(D^0) - d(\sigma)$, the projection map from $\clM_{\sigma}$ to $\clH$ is Fredholm of index $-d(\sigma)$. The proposition is proved because of the separability of $\clH_0\times (A\cap L_2)$.
\end{proof}

\begin{Corollary}
\label{cor_equi_trans}
Let $L_1$ and $L_2$ be two $G$--invariant Lagrangians in $(M, \omega)$ such that their intersection is contained in a compact subset of $M$. Let $\clH_{0}$ be a small neighborhood of $0$ in $\clH$.Then	for a generic perturbation $H_{s} \in\clH_0$, the pair $\Phi_{H_{s}}(L_1)$ and $L_2$ intersect non-degenerately.
\end{Corollary}

\begin{proof}
Using Proposition \ref{prop_universal_moduli_space}, because $d(\sigma) > 0$ as long as $\sigma$ is not given by the $0$--representation, $\clH_{\sigma} := \pi_{\sigma}(\clM_{\sigma})$ is a $C^{\infty}$--subvariety of $\clH^{0}$ of positive codimension. As a result, 
$$\bigcup_{\sigma \in \clR_G} \clH_{\sigma}$$
is a meager subset of $\clH_{0}$. Any Hamiltonian perturbation constructed from elements in $\clH_0 \backslash \bigcup_{\sigma \in \clR_G} \clH_{\sigma}$ establishes the transversality result.
\end{proof}

Now we show that for a generic $1$--paramter family $\{ H_{s,t} \}_{0 \leq t \leq 1}$, the intersection $\Phi_{H_{s,t}}(L_1) \cap L_2$ has at most one degenerate intersection orbit which corresponds to an irreducible representation for each $t$.

Choose $\clH_0$ and $A$ as before. Suppose $\sigma_1, \sigma_2 \in \clR_{G}$. Define $\clM_{\sigma_1, \sigma_2} \subset \clH_0 \times (A \cap L_2) \times (A \cap L_2)$ to be the set of points $(H_{s}, p_1, p_2)$ such that 
\begin{enumerate}
	\item $p_1, p_2 \in \Phi_{H_{s}}(L_1)\cap L_2$,
     \item $\Orb(p_1) \cap \Orb(p_2) = \emptyset$,
	\item The spaces $\big(T_{p_i} \Phi_{H_{s}}(L_1) \cap T_{p_i}(L_2)\big) / T_{p_i}\Orb(p_i)$ as $\Stab(p_i)$--representations represent $\sigma_i$ for $i = 1, 2$.
\end{enumerate}
Let $\clH_{\sigma_1, \sigma_2}$ be the projection of $\clM_{\sigma_1, \sigma_2}$ onto $\clH_0$. The following statement follows from the arguments in the proof of Proposition \ref{prop_universal_moduli_space} and \cite[Lemma 3.23]{bai2020equivariant}, thus the proof is omitted.

\begin{Proposition}
\label{prop_2rep}
$\clH_{\sigma_1, \sigma_2}$ is a $C^{\infty}$--subvariety of $\clH_0$ of codimension at least $d(\sigma_1) + d(\sigma_2)$. \qed
\end{Proposition}

Note that $d(\sigma) = 1$ if and only if $\sigma$ is irreducible. Therefore the following corollary follows immediately from Proposition \ref{prop_universal_moduli_space} and Proposition \ref{prop_2rep} by an application of the Sard-Smale theorem.

\begin{Corollary}
Let $\{ H_{s,t} \}_{0 \leq s \leq 1}$ be a generic path in $\clH_0$ connecting $H_{s,0}, H_{s,1} \in \clH_0 \backslash \bigcup_{\sigma \in \clR_G} \clH_{\sigma}$ which intersect all $\clH_{\sigma}$ and $\clH_{\sigma_1, \sigma_2}$ transversely. Then there are at most countably many $t$ such that $\Phi_{H_{s,t}}(L_1)$ do not intersect $L_2$ non-degenerately. For every such $t \in (0,1)$, there is exactly one intersection orbit $\Orb(p_t)$ such that $\big(T_{p_t} \Phi_{H_{s,t}}(L_1) \cap T_{p_t}(L_2)\big) / T_{p_t}\Orb(p_t)$ is an irreducible $\Stab(p_t)$--representation, and  $\Phi_{H_{s,t}}(L_1)$ intersect $L_2$ non-degenerately along any other orbits. \qed
\end{Corollary}

\subsection{An equivariant Maslov index}
\label{subsec_equiv_mas_ind}

In this subsection, we give the definition of equivariant Maslov index which will appear in the index formula later on.

We recall the definition of the classical Maslov index. Let $(V_0, \omega_0)$ be a symplectic vector space. For $a, b \in \bR$ such that $a < b$, let $\clP_{[a,b]}(V_0)$ be the space of continuous and piecewise smooth maps
$$q: [a, b] \rightarrow \{ \textrm{paris of Lagrangian subspaces in }  V_0 \}.$$
Then there exists a map
$$\mu_{V_0}: \clP_{[a,b]}(V_0) \rightarrow \bZ$$ which is called the \emph{Maslov index} uniquely characterized by the following properties (see \cite{cappell1994maslov}):
\begin{enumerate}
\item (Affine scale invariance) Suppose $k > 0$ and $l \geq 0$ and let $\Gamma: [a, b] \rightarrow [ka+l, kb+l]$ be the affine linear map given by $t \mapsto kt+l$. Then for $q \in \clP_{[ka+l, kb+l]}(V_0)$, we have
$$\mu_{V_0}(q) = \mu_{V_0}(q \circ \Gamma);$$
\item (Homotopy invariance) Given $q_0, q_1 \in \clP_{[a,b]}(V_0)$ such that $q_0(a) = q_1(a)$ and $q_0(b) = q_1(b)$, if there is a homotopy between $q_0$ and $q_1$ in $\clP_{[a,b]}(V_0)$ relative to end points, then
$$\mu_{V_0}(q_0) = \mu_{V_0}(q_1);$$
\item(Additivity under concatenation of paths) If $a < c < b$ and $q \in \clP_{[a,b]}(V_0)$, then the following additive property is satisfied:
$$\mu_{V_0}(q) = \mu_{V_0}(q|_{[a,c]}) + \mu_{V_0}(q|_{[c,b]});$$
\item (Symplectic additivity) Let $(V_0^{\prime}, \omega_0^{\prime})$ be another symplectic vector space and suppose $q \in \clP_{[a,b]}(V_0), q' \in \clP_{[a,b]}(V_0^{\prime})$. Then
$$\mu_{V_0 \oplus V_0^{\prime}}(q \oplus q') = \mu_{V_0}(q) + \mu_{V_0^{\prime}}(q');$$
\item (Independence of symplectic framing) Suppose $\phi_t: [a,b] \rightarrow \Sp(V_0)$ is a $1$--parameter family of symplectic matrices. Denote by $q(t) = (L_1(t), L_2(t))$ for some $q \in \clP_{[a,b]}(V_0)$ and define $\phi_{*}(q)(t) = (\phi_t (L_1(t)), \phi_t (L_2(t)))$. Then
$$\mu_{V_0}(\phi_{*}(q)(t)) = \mu_{V_0}(q(t));$$
\item (Normalization) Equip $\bR^2_{x, y} \cong \bC_{z=x+iy}$ with the symplectic form $dx \wedge dy = \frac{i}{2} dz \wedge d \overline{z}$. Consider the path of Lagrangians $p(t)$ in $\clP_{[-\frac{\pi}{4}, \frac{\pi}{4}]}(\bR^2)$ given by
$$t \mapsto (\bR\{1\}, \bR\{e^{it}\}).$$
Then
$$\mu_{\bR^2}(p|_{[-\frac{\pi}{4}, \frac{\pi}{4}]}) = 1, \mbox{ } \mu_{\bR^2}(p|_{[-\frac{\pi}{4}, 0]}) = 0, \mbox{ } \mu_{\bR^2}(p|_{[-\frac{\pi}{4}, 0]}) = 1.$$
\end{enumerate}

There are many ways to construct the map $\mu_{V_0}$ explicitly. For our purpose later, we use the characterization of $\mu_{V_0}$ as the spectral flow of certain self-adjoint operators, see \cite[Section 7]{cappell1994maslov}. Let $J_0 \in GL(V_0)$ be a complex structure which is compatible with $\omega_0$. Given $L_1, L_2 \subset V_0$ which are linear Lagrangian subspaces, let $L^2_{1}([0,1]; L_1, L_2)$ be the completion of smooth maps $f: [0,1] \rightarrow V$ with $f(0) \in L_1, f(1) \in L_2$ under $L^2_1$--norm induced by the metric $\omega(\cdot, J_0 \cdot)$. Let $L^2([0,1]; V_0)$ be the space of $L^2$--maps from $[0,1]$ to $V_0$. Then the operator
$$D_{L_1, L_2}: L^2_{1}([0,1]; L_1, L_2) \rightarrow L^2([0,1]; V_0)$$
$$f \mapsto - J_0 \frac{df}{dt}$$
is a real self-adjoint unbounded Fredholm operator with discrete spectrum. The kernel of $D_{L_1, L_2}$ is the same as constant functions taking value in $L_1 \cap L_2$.

If we take a path of pairs of Lagrangian subspaces $q(t) = (L_1(t), L_2(t))_{t \in [a,b]}$ in $V_0$, we have a family of self-adjoint operators $\{ D_{L_1(t), L_2(t)} \} _{a \leq t \leq b}$. Let $\epsilon > 0$ be small enough so that $D_{L_1(a), L_2(a)}$ and $D_{L_1(b), L_2(b)}$ do not have any other eigenvalues between $- \epsilon$ and $\epsilon$ except for $0$.
\begin{Proposition}
\label{prop_maslov_spectral}
The spectral flow of \{$D_{L_1(t), L_2(t)} - \epsilon \id \}_{a \leq t \leq b}$ is independent of the choice of $J_0$ and equal to $\mu_{V_0}(p).$ \qed
\end{Proposition}

Now suppose $V_0$ is a $H$--representation where $H$ is a connected compact Lie group, and the $H$--action preserves the symplectic form $\omega_0$. Let 
$$V_0 \cong (V_0^{1})^{\oplus a_1} \oplus \cdots \oplus (V_0^{m})^{\oplus a_m}$$
be the isotypic decomposition of $V_0$. Then each isotypic piece inherits a symplectic structure and the above decomposition is compatible with the symplectic structure.
If $q(t) = (L_1(t), L_2(t))_{t \in [a,b]}$ is a path such that both $L_1(t)$ and $L_2(t)$ are preserved under the $H$--action, then
$$(V_0^{1})^{\oplus a_1} \cap L_i(t) \oplus \cdots \oplus (V_0^{m})^{\oplus a_m} \cap L_i(t)$$
gives the isotypic decomposition of $L_{i}(t)$ for $i=1,2$ and $t \in [a,b]$. Write $q_{k}(t)$ as the induced path on the isotypic piece $(V_0^{k})^{\oplus a_k}$ for $1 \leq k \leq m$. 
\begin{Definition}
The \emph{equivariant Maslov index} $\mu^{H}(q)$ of the path $q(t)$ as above is the element
\begin{equation}
\sum_{i=1}^m \frac{1}{\dim_{\bR} V_i} \cdot \mu_{(V_0^{i})^{\oplus a_i}}(q_{i}) \cdot [V_i] \in \clR(H).
\end{equation}
\end{Definition}

\begin{remark}
The above quantity is well-defined because the $V_i$--isotypic piece of $L_1(t) \cap L_2(t)$ would change the Maslov index by a multiple of $\dim_{\bR} V_i$. Alternatively, we can define $\mu^{H}(q)$ using the equivariant spectral flow of the $H$--equivariant operator $-J_{0} \frac{d}{dt}$ where $J_0$ is an $H$--equivariant complex structure on $V_0$ compatible with $\omega_0$, based on Definition \ref{def_equi_spectral_flow}  and Proposition \ref{prop_maslov_spectral}.
\end{remark}

Now we go back to discuss about global symplectic geometry. Recall that $(M, \omega)$ is a symplectic manifold with a symplectic $G$--action and $L_1$ and $L_2$ are two $G$--invariant Lagrangians of $M$. Suppose $H$ is a closed subgroup of $G$. Let $\overline{\mathbb{D}} \subset \bC_{z} = \bR^2_{x,y}$ be the closed unit disc in the plane. Suppose we have a map $u: \overline{\mathbb{D} }\rightarrow M$ which is smooth except at $\pm 1$ satisfying:
\begin{enumerate}
\item $u(\partial \overline{\mathbb{D}} \cap \{ y \leq 0 \}) \subset L_1$ and $u(\partial \overline{\mathbb{D}} \cap \{ y \geq 0 \}) \subset L_2$;
\item For any $z \in \overline{\mathbb{D}}$, the Lie group $\Stab(u(z))$ contains $H$ as a subgroup.
\end{enumerate}
Choose $p = u(-1)$ as the base point. Then the symplectic vector bundle $u^{*} TM$ is isomorphic to the product bundle $T_{p}M \times \overline{\mathbb{D}}$ respecting the fiberwise symplectic $H$--action under a trivialization $\Psi$. Using $\Psi$, 
$$q_u = (u^{*}|_{\partial \overline{\mathbb{D}} \cap \{ y \leq 0 \}} TL_1, u^{*}|_{\partial \overline{\mathbb{D}} \cap \{ y \geq 0 \}} TL_2)$$
defines a path of pair of $H$--invariant Lagrangian subspaces in $T_{p}M$, after identifying $\partial \overline{\mathbb{D}} \cap \{ y \leq 0 \}$ and $\partial \overline{\mathbb{D}} \cap \{ y \geq 0 \}$ with the interval $[-1, 1]$ by projecting to the $x$--coordinate. Recall that $\mathfrak{h}$ is the Lie algebra of $H$. Because $H \subset \Stab(u(z))$ for all $z \in \overline{\mathbb{D}}$, the symplectic slice theorem implies that $T_{u(z)}M$ contains the symplectic subspace $T^{*}(\mathfrak{g} / \mathfrak{h})$, i.e. the tangent space of the cotangent bundle of $\Orb(u(z))$. The tangent spaces of $T_{u(z)} L_i, z \in \partial \mathbb{D}$ contain a copy of $\mathfrak{g} / \mathfrak{h}$ as a result of the slice theorem. Define the symplectic vector space
$$\overline{T_p M} = T_p M / T^{*}(\mathfrak{g} / \mathfrak{h})$$
with the induced symplectic structure and the path of pairs of linear Lagrangian subspaces
$$\overline{q}_u = (u^{*}|_{\partial \overline{\mathbb{D}} \cap \{ y \leq 0 \}} TL_1 / (\mathfrak{g} / \mathfrak{h}) , u^{*}|_{\partial \overline{\mathbb{D}} \cap \{ y \geq 0 \}} TL_2 / (\mathfrak{g} / \mathfrak{h}))$$
which lies inside $\overline{T_p M}$.

\begin{Definition}
\label{def_equi_mas}
Define the \emph{$H$--equivariant Maslov index} $\mu^{H}(u) \in \clR(H)$ of the map $u$ as above to be
\begin{equation}
\mu^{H}(u) := \mu^{H}(\overline{q}_u).
\end{equation}
\end{Definition}
Using the properties of the Maslov index, it is easy to see that $\mu^{H}(u)$ is independent of the trivialization $\Psi$. It is also invariant under the homotopies of $u$ which satisfy (1) and (2) above and keep both $u(-1)$ and $u(1)$ fixed.

\begin{example}
\label{ex_local_model}
Suppose $V$ is an orthogonal representation of $G$, and suppose $(M,\omega)$ is given by the cotangent bundle of $V$. Let $f$ be a smooth $G$--invariant function on $V$ with finitely many critical orbits. Furthermore, we assume that $f$ is a $G$--Morse function, i.e. the Hessian of $f$ is non-degenerate on $T_p V / T_p \Orb(p)$ if $p$ is a critical point of $f$. Let $L_1$ be the graph of the zero section of $M=T^*V$, let $L_2$ be the graph of $df$. Then it is easy to see that $L_1$ and $L_2$ intersect non-degenerately. Let $p \in V$ be critical point of $f$ such that $H \subset \Stab(p)$. Let 
$$V / T \Orb(p) = (V_1)^{\oplus a_1} \oplus \cdots \oplus (V_m)^{\oplus a_m}$$
be the isotypic decomposition of $V$ as an $H$--representation. Then for each $1 \leq i \leq m$, the Hessian of $f$ induces an invertible self-adjoint $H$--equivariant linear map
$$(V_i)^{\oplus a_i} \mapsto (V_i)^{\oplus a_i}$$
whose negative eigenspace has dimension denoted by $\dim_{\bR}V_i \cdot \ind(p, V_i)$. Define
$$\ind_{H}(p) = \sum_{i=1}^{m} \ind(p, V_i) \cdot [V_i] \in \clR(H).$$

Now consider that we have a pair of critical points $p_1$ and $p_2$ of $f$ which do not lie on the same $G$--orbit and they are both fixed under the action of $H$. Then we can find a path $\gamma$ in $V$ which is invariant under the $H$--action connecting $p_1$ and $p_2$. Note that the time-$t$ Hamiltonian flow of the function $f: T^{*}V \rightarrow \bR$ is given by
$$(v, v') \rightarrow (v, v' + t \cdot df),$$
where we extend $f: V \rightarrow \bR$ to a smooth function over $T^{*}V$ which is constant along each fiber. The trace of the image of $\gamma$ for $t \in [0,1]$ defines a map $u_{\gamma}: \overline{\mathbb{D}} \rightarrow T^{*}V$ which has a well-defined Maslov index. It could be shown that
$$\mu^{H}(u_{\gamma}) =\ind_{H}(p_2) - \ind_{H}(p_1)$$
which is a classical computation of Lagrangian Floer homology of cotangent bundles.
\end{example}

Example \ref{ex_local_model} shows that in this simple case, the non-degeneracy of critical points of a $G$--invariant smooth function is equivalent to the non-degeneracy of intersection between a closely-related pair of Lagrangians. Moreover, the equivariant Maslov index could be computed as a relative equivariant Morse index. We extend these basic correspondences in the next section to infinite dimensions which eventually allow us to interpret the generalized Casson invariants using finite-dimensional symplectic constructions.

%% file: index.tex
This section proves the most important technical results of this paper. Given a three manifold $Y$ equipped with a Heegaard splitting $Y = H_{1} \cup_{\Sigma} H_2$, we show that there is a bijection between gauge equivalence classes of perturbed flat $G$--connections on $Y$ and intersection $G$--orbits of Lagrangians in $\hat{\mathcal{M}}^{\mathfrak{g}}(\Sigma')$ induced by $H_1$ and $H_2$ after applying a compatible $G$--equivariant Hamiltonian isotopy. Moreover, the non-degeneracy condition from gauge-theoretic context is shown to be equivalent to the non-degeneracy condition of Lagrangian intersections. We then prove several technical results which eventually reduce the computation of spectral flows on $Y$ to the computation of Maslov indices on $\hat{\mathcal{M}}^{\mathfrak{g}}(\Sigma')$.

\subsection{Perturbations}
\label{subsec_pert}
Let $Y$ be a closed, oriented, Riemannian 3-manifold with a splitting 
$$Y=H_1\cup_{\Sigma} \big([-2,2]\times \Sigma\big) \cup_\Sigma H_2,$$
 where $H_1$ and $H_2$ are handlebodies and the genus of $\Sigma$ is $h \geq 2$. Assume further that the metric on $[-2,2]\times\Sigma$ is cylindrical. Let $\{p\} \in \Sigma$, $D^2_p$ and $\Sigma'$ be the same as in Section \ref{sec_Extended}. Throughout this section, the product principal $G$--bundle over any manifold will be denoted by $P$.
We introduce a family of cylinder functions on $Y$ supported on $[-2,2]\times \Sigma$, which are closely related to Hamiltonian perturbations on the extended moduli space $\hat{\mathcal{M}}^{\mathfrak{g}}(\Sigma')$.
 
Suppose 
 $$
 \gamma_j:(\bR/\bZ)\times(-1,1)\to \Sigma' \subset \Sigma, \qquad j=1,\cdots,m
 $$
 are smooth immersions, and suppose there exists $\epsilon>0$ such that the $\gamma_j$'s coincide on $(-\epsilon,\epsilon)\times (-1,1)$.  Let 
 $$
 h_s: G^m\to \bR, \qquad s\in[-1,1]
 $$
 be a $1$--parameter family of smooth functions on $G^m$ invariant under diagonal conjugations by $G$, such that $h_s$ vanishes when $s$ is near $-1$ and $1$. 
Let 
$$\beta:(-1,1)\to \bR$$
be a smooth, compactly supported function that integrates to $1$. Define the function $H_s:\clA(\Sigma) \mbox{ or } \clA(\Sigma') \to \bR$ by
\begin{equation}
\label{eqn_ham}
H_s(A) = \int_{-1}^1 \beta(\lambda) h_s(\rho_\lambda(A))\,d\lambda
\end{equation}
for $s\in[-1,1]$. 
It is easy to see that $H_s$ is invariant under the gauge group actions. Using the affine linear map $x \mapsto 2x - 1$ identifying $[0,1]$ and $[-1,1]$, we can use the coordinate $\hat{s}$ on $[0,1]$ to get the family $\{ H_{\hat{s}} \}_{0 \leq \hat{s} \leq 1}$ from the family $\{ H_{s} \}_{-1 \leq s \leq 1}$. This fits in with the convention set up in Section \ref{sec_Lag_int}. The Hamiltonian $H_{\hat{s}}$ descends to a $G$--invariant time-dependent Hamiltonian on the moduli space $\hat{\mathcal{M}}^{\mathfrak{g}}(\Sigma')$. The induced Hamiltonian flow preserves the set $\mu^{-1}(0) \subset \hat{\mathcal{M}}^{\mathfrak{g}}(\Sigma')$.

\begin{Lemma}
\label{lem_surj}
Every time-dependent smooth $G$--invariant Hamiltonian on $\hat{\mathcal{M}}^{\mathfrak{g}}(\Sigma')$ could be constructed in this way.
\end{Lemma}
\begin{proof}
Choose $\{ \alpha_i, \beta_i \}_{1 \leq i \leq h}$ to be loops with common base point in $\Sigma'$ such that they generate the fundamental group of $\Sigma'$. They could be thickened into immersions of annuli of the form described above. For any $[A] \in \hat{\mathcal{M}}^{\mathfrak{g}}(\Sigma')$, it is uniquely determined by its holonomies around $\alpha_i, \beta_i, 1 \leq i \leq h$. This is the embedding described in Lemma \ref{lem_trivial}. Thus all the $G$--invariant functions on $\hat{\mathcal{M}}^{\mathfrak{g}}(\Sigma')$ all come from restricting conjugation invariant functions on $G^{2h}$. Given a smooth $1$--parameter family of conjugation invariant functions on $G^{2h}$, denoted by $h_{\hat{s}}$, we can construct the function $H_{\hat{s}}$ using $h_{\hat{s}}$, the annuli and the cut-off function $\beta$. It is easy to see that the function on $\hat{\mathcal{M}}^{\mathfrak{g}}(\Sigma')$ induced from $H_{\hat{s}}$ is the same as $h_{\hat{s}}$. This proves the lemma.
\end{proof}

On the other hand, the maps $\{ \gamma_j \}_{1 \leq j \leq m}$ could be used to define smooth immersions
$$q_j: (\bR/\bZ)\times(-1,1)^2 \rightarrow \Sigma' \times [-1, 1] \qquad j=1,\cdots,m$$
simply by requiring that $q_j|_{\Sigma' \times \{s\}} = \gamma_j$ for all $1 \leq j \leq m$. Note that the tuple
$${\bf q} = (q_1,\cdots,q_m,\beta,h_s)$$
defines a cylinder datum on $\Sigma' \times [-1,1] \subset \Sigma \times [-1,1]$ as in Definition \ref{def_cyl_datum}. Because $h_s$ vanishes for $s$ near $\pm 1$, we see that ${\bf q}$ defines a cylinder datum on $Y$. Therefore we have the function
$$f_{\bf q}: \clC_{Y} \rightarrow \bR$$
which defines a holonomy perturbation.

\begin{Definition}
The Hamiltonian $H_{\hat{s}}:\hat{\mathcal{M}}^{\mathfrak{g}}(\Sigma') \rightarrow \bR$ and the cylinder function $f_{\bf q}: \clC_{Y} \rightarrow \bR$ constructed as above are called a \emph{compatible pair}.
\end{Definition}

Suppose $A\in \clC(Y)$ is in temporal gauge on $[-1,1]\times \Sigma$, i.e. $\iota(\frac{\partial}{\partial s})(A) = 0$ where $s$ is the coordinate on $[-1,1]$. It is well-known that every connection is gauge-equivalent to a connection in temporal gauge on $[-1,1]\times \Sigma$ in this case. Let $A_s$ be the restriction of $A$ to the slice $\{s\} \times \Sigma$.
Then $A \in \clC(Y)$ is a critical point of the perturbed Chern-Simons functional $\CS + f_{\bf q}$ if and only if
\begin{equation}
\label{eqn_flat_comple}
A \mbox{ is flat on the complement of } [-1,1]\times \Sigma,
\end{equation}
\begin{equation}
\label{eqn_flat_surface}
F_{A_s} = 0 \mbox{ for } s\in[-1,1],
\end{equation}
\begin{equation}
\label{eqn_ham_chord}
\dot A - X_s(A) = 0 \mbox{ for } s\in[-1,1],
\end{equation}
where $X_s(A)$ is a smooth map $X_s:\clA(\Sigma) \to \Omega^1(\Sigma) \otimes \mathfrak{g}$ such that
\begin{equation}
\label{eqn_vector_field}
dH_s(A) a= \int_{\Sigma}\Tr (a \wedge X_s(A)).
\end{equation}
Notice that $X_s(A)\in \Omega^1(\Sigma) \otimes \mathfrak{g}$ is supported in the union of $ \ima(\gamma_j)$.

Notice that $H_1\cup [-2,-1]\times \Sigma$ and $[1,2] \times H_2$ are manifolds with boundary both given by $\Sigma$. Let $L_1,L_2\subset \mu^{-1}(0)$ be the corresponding $G$--invariant Lagrangian submanifolds of $\hat{\mathcal{M}}^{\mathfrak{g}}(\Sigma')$ as described in Section \ref{sec_Extended}. Recall that $\Phi_{H_{\hat{s}}}$ is the time $1$--map of the Hamiltonian flow induced by $H_{\hat{s}}$. Now we are ready to establish the following set-theoretic correspondence.

\begin{Proposition}
\label{prop_set}
Let $H_{\hat{s}}$ and $f_{\bf q}$ be a compatible pair. Then there is a one-to-one correspondence between the orbits of perturbed flat connections on $Y$ with respect to $f_{\bf q}$ and the $G$--orbits of $\Phi_{H_{\hat{s}}}(L_1)\cap L_2$. Moreover, a perturbed flat connection is non-degenerate if and only if the corresponding intersection of $\Phi_{H_{\hat{s}}}(L_1)$ and $L_2$ is non-degenerate.
\end{Proposition}

\begin{proof}
Throughout the proof, we write $\Phi_{H_s} = \Phi_{H_{\hat{s}}}$ and work with the $s$--coordinate.

	Suppose $A_s$ with $-1 \leq s \leq 1$ is a smooth 1-parameter family in $\clA(\Sigma)$ that satisfies equation \eqref{eqn_flat_surface} and equation \eqref{eqn_ham_chord}.
Because $A_s$ is flat, without loss of generality we can assume that $A_s |_{D^2_p}$ is given by the product connection after applying a gauge transformation. Then the image $[A_s] \in \hat{\mathcal{M}}^{\mathfrak{g}}(\Sigma')$ defines a smooth path.
Recall that $H_s$ is a family of smooth Hamiltonian functions on $\hat{\mathcal{M}}^{\mathfrak{g}}(\Sigma')$. Then $[A_s]$ is a Hamiltonian chord of $H_s$ by equation \eqref{eqn_vector_field} and the description of the symplectic form $\omega_{\Sigma'}$. Now if $A$ is an $f_{\bf q}$--perturbed flat connection, by equation \eqref{eqn_flat_comple} $[A_{-1}]$ (resp. $[A_1]$) actually lies in $L_1$ (resp. $L_2$). By equivariance we construct a $G$--orbit in $\Phi_{H_s}(L_1)\cap L_2$.

	On the other hand, suppose $[A_s]$ is a Hamiltonian chord of $H_s$ on $\hat{\mathcal{M}}^{\mathfrak{g}}(\Sigma') \cap \mu^{-1}(0)$. Let $A_s$ be a lift of $[A_s]$ to $\clA^0_F(\Sigma)$, namely a family of flat connections on $\Sigma$ which are trivial over the disc $D_p^2$, then we have
	$$
	\int_\Sigma \Tr(\dot A_s - X_s(A_s), a) = 0
	$$
	for any $s$ and any $a \in \Omega^1(\Sigma) \otimes \mathfrak{g}$ such that $d_{A_s} a=0$. Therefore there exists $\phi_s$ such that
	$$
	\dot A_s - X_s(A_s) = d_{A_s}\phi_s.
	$$ 
	Hence $A_s$ solves the equations \eqref{eqn_flat_surface} and \eqref{eqn_ham_chord} after a 1-parameter family of gauge transformations. The Lagrangian boundary condition lifts to \eqref{eqn_flat_comple}.
	This proves the first part of the statement.
	
	For the second part, suppose $A \in \clC(Y)$ is a degenerate perturbed flat connection, then there exists $v \in \ker d^*_A \backslash \{0\}$ such that $K_{A,f_{\bf q}} v = 0$. We have $[A_{-1}] \in L_1$, $[A_{1}] \in L_2$, and $\Phi_{H_s}(\bar A_{-1}) = A_{1}$.   Restricting $v$ to $\{\pm 1\}\times\Sigma$ defines twisted harmonic $1$--forms in $H^{1}_{A_{\pm1}}$ denoted by $v_{\pm 1}$. Under the identification 
$$ T_{[A_{\pm 1}]} \hat{\mathcal{M}}^{\mathfrak{g}}(\Sigma') \cong H_{A_{\pm 1}}^1 \oplus T^*(\mathfrak{g}/ H_{A_{\pm 1}}^0) $$
from Corollary \ref{cor_slice}, it is easy to see that $v_{-1}$ and $v_1$ define tangent vectors in $T_{[A_1]} L_1$ and $T_{[A_1]} L_2$ which are orthogonal to the directions corresponding to the $G$--orbits respectively. By linearizing equation \eqref{eqn_ham_chord} and integrate along the $s$--direction, we see that $(\Phi_{H_s})_{*}(v_{-1}) = v_1$. Therefore the intersection of $\Phi_{H_s}(L_1)$ and $L_2$ is degenerate at $[A_1]$.
	
	On the other hand, suppose the intersection of $\Phi_{H_s}(L_1)$ and $L_2$ is degenerate, then we obtain a Hamiltonian chord $[A_s]$ of $H_s$ connecting $L_1$ and $L_2$ and there exists $v_{\pm 1}\in T_{[A_{\pm 1}]} L_{\pm 1}$ such that $v_{\pm 1} \notin T\Orb([A_{\pm 1}])$ and $(\Phi_{H_s})_{*}(v_{-1}) = v_{+1}$. Therefore, we can identify $v_{\pm 1}$ as elements in $H_{A_{\pm 1}}$. Lift $v_{\pm 1}$ to $w_{\pm 1} \in \Omega^1(\Sigma) \otimes \frg$, then $w_\pm$ extends to $w \in \Omega^1(Y) \otimes \mathfrak{g}$ such that $w \in \ker ( *d_A + dV_{f_{\bf q}}(A) )$ and $w\notin \ima d_A$. Decompose $w$ as $w = d_A u + w'$ where $d_A^* w' = 0$, we have $w' \in \Hess K_{A,f_{\bf q}}$, therefore $A$ is degenerate as a perturbed flat connection.
\end{proof}

Using Corollary \ref{cor_equi_trans}, Lemma \ref{lem_surj} and Proposition \ref{prop_set}, we conclude

\begin{Proposition}
\label{prop_trans_equal}
For a generic choice of small Hamiltonian $H_{\hat{s}}: \clA(\Sigma) \rightarrow \bR$ constructed as above, the intersection between $G$--equivariant Lagrangians $\Phi_{H_{\hat{s}}}(L_1)$ and $L_2$ is non-degenerate. Accordingly, the perturbation of the Chern-Simons functional $\CS$ on Y induced by $f_{\bf q}$ constructed as above is non-degenerate. \qed
\end{Proposition}

\begin{remark}
The above statement shows that we can obtain equivariant transversality of $\CS$ by a perturbation supported in the neck region $[-1, 1] \times \Sigma$ instead of using loops wiggling all over the $3$--manifold $Y$. Of course, this would not be a surprise because the inclusion $\Sigma \hookrightarrow Y$ defines a surjection between fundamental groups so a generic perturbation of the character variety of $Y$ should result form a generic perturbation of it as a subset in the character variety of $\Sigma$.
\end{remark}

\subsection{Spectral flow comparison}
\label{subsec_spec_comp}

Let $Y=H_1\cup_{\Sigma} \big([-2,2]\times \Sigma\big) \cup_\Sigma H_2$ and $P$ be the same as before. Suppose $H$ is a closed subgroup of the structural group $G$. We choose a generic Hamiltonian and the corresponding compatible holonomy perturbation $f_{\bf q}$ as from Proposition \ref{prop_trans_equal}. Suppose $B_0, B_1 \in \clC(Y)$ are two critical points of $\CS + f_{\bf q}$ such that $ H = \Stab(B_1) \subset \Stab(B_0)$ and let $A_{s,0}, A_{s,1}$ be the corresponding path in $\clA(\Sigma)$ satisfying equations \eqref{eqn_flat_surface} and \eqref{eqn_ham_chord}. Let us further assume that we actually have a smooth $2$--parameter family
$$A_{s,t} \in \clA(\Sigma), \quad (s,t) \in [-1,1] \times [0,1]$$
connecting the $1$--parameter families $A_{s,0}$ and $A_{s,1}$ such that the following holds:
\begin{equation}
F_{A_{s,t}} = 0 \mbox{ for } (s,t) \in [-1,1] \times [0,1];
\end{equation}
\begin{equation}
\label{eqn_extension}
\begin{aligned}
\mbox{For all } t \in [0,1],  A_{-1, t} \mbox{ and } A_{1,t} \mbox{ extend to flat connections on} \\ 
\mbox{$H_1 \cup [-2,-1] \times \Sigma$ and $[1,2] \times \Sigma \cup H_2$ respectively.}
\end{aligned}
\end{equation}
Using the above two conditions, we can construct a $1$--parameter family of connections $\{B_t\}_{0 \leq t \leq 1}$ on $Y$ connecting $B_0$ and $B_1$ which are in temporal gauge on $[-1,1] \times \Sigma$ by extending $A_{s,t}$. We further assume that
\begin{equation}
H \subset \Stab(B_t) \mbox{ for all } t \in [0,1].
\end{equation}
Then we have a $1$--parameter family of self-adjoint operators $K_{B_t, f_{\bf q}}$ and it has a well-defined $H$--equivariant spectral flow $Sf_{H}(K_{B_t, f_{\bf q}})$ as in Definition \ref{def_equi_spectral_flow}. By applying a homotopy and reparametrization, we can assume that for all $t \in [0,1]$, the connection $B_t$ is in temporal gauge on $[-2, 2] \times \Sigma$ which defines the family $\{ A_{s,t} \}_{-2 \leq s \leq 2} \subset \clA_{\Sigma}$ and this family is constant (i.e. independent of $s$) over $[-2,1] \times \Sigma$ and $[1,2] \times \Sigma$.

On the other hand, we can consider the space 
$$\Omega^1(\Sigma) \otimes \mathfrak{g} \oplus  \Omega^0(\Sigma)  \otimes \mathfrak{g} \oplus  \Omega^2(\Sigma)  \otimes \mathfrak{g}.$$ It comes with a natural symplectic form
$$\tilde{\omega}_{\Sigma}: ((a_1, \phi_1, \psi_1),(a_2, \phi_2, \psi_2)) \mapsto \int_{\Sigma} \Tr(a_1 \wedge a_2) + \int_{\Sigma} \Tr(\phi_1 \wedge \psi_2) - \int_{\Sigma} \Tr(\psi_1 \wedge \phi_2)$$ 
and a compatible complex structure $\tilde{J}_{\Sigma}$ given by
$$(a, \phi, \psi) \mapsto (*a, -*\psi, *\phi).$$
Given $A \in \clA_{F}(\Sigma)$, we define $S_A$ to be the \emph{twisted de Rham operator} on $\Omega^1(\Sigma) \otimes \mathfrak{g} \oplus  \Omega^0(\Sigma)  \otimes \mathfrak{g} \oplus  \Omega^2(\Sigma)  \otimes \mathfrak{g}$ given by
$$(a, \phi, \psi) \mapsto (-*d_A \phi - d_A * \psi, *d_A a, d_A * a).$$
This is a self-adjoint operator with respect to the natural $L^2$--norm.
Now consider the space $$L_k^2([-1,1]; \Omega^1(\Sigma) \otimes \mathfrak{g} \oplus  \Omega^0(\Sigma)  \otimes \mathfrak{g} \oplus  \Omega^2(\Sigma)  \otimes \mathfrak{g})$$
which is the $L_k^2$--completion of the space of smooth maps from $[-1,1]$ to $\Omega^0(\Sigma) \otimes \mathfrak{g} \oplus  \Omega^1(\Sigma)  \otimes \mathfrak{g} \oplus  \Omega^2(\Sigma)  \otimes \mathfrak{g}$ using the cylindrical metric on $[-1,1] \times \Sigma$. Note that this space could be identified with 
$$L_k^2(\Omega^0([-1,1] \times \Sigma) \otimes \mathfrak{g}) \oplus L_k^2(\Omega^1([-1,1] \times \Sigma)  \otimes \mathfrak{g})$$
on $[-1,1] \times \Sigma$ by mapping $(a, \phi, \psi) \in L_k^2([-1,1]; \Omega^1(\Sigma) \otimes \mathfrak{g} \oplus  \Omega^0(\Sigma)  \otimes \mathfrak{g} \oplus  \Omega^2(\Sigma)  \otimes \mathfrak{g})$ to $(\phi, a + *\psi ds)$. Because the holonomy perturbation $f_{\bf q}$ is supported on $[-1,1] \times \Sigma$, for each $t \in [0,1]$, the operator $K_{B_t, f_{\bf q}}$ is equal to the twisted odd signature operator
\begin{equation}
\label{eqn_twisted_sign}
\begin{pmatrix}
 	*d_{B_t} & d_{B_t} \\
 	d_{B_t}^* & 0
 \end{pmatrix}
\end{equation}
on $(\Omega^0(Y) \oplus \Omega^1(Y)) \otimes \mathfrak{g}$ in the complement of $[-1,1] \times \Sigma$. The restrictions of this operator on $[-2,-1] \times \Sigma$ and $[1,2] \times \Sigma$ are all of the form 
$$* (ds \wedge \frac{\partial}{\partial s} + S_{A_{\mp 1, t}})$$
under the identification described above. Let $\Lambda_{-1}(t) \subset L^2(\Omega^1(\Sigma) \otimes \mathfrak{g} \oplus  \Omega^0(\Sigma)  \otimes \mathfrak{g} \oplus  \Omega^2(\Sigma)  \otimes \mathfrak{g})$ be the subspace consisting of elements that extend to $L^{2}_{1/2}$ elements in $\ker(K_{B_t, f_{\bf q}})$ on $H_1 \cup [-2,-1] \times \Sigma$ and let $\Lambda_{1}(t) \subset L^2(\Omega^1(\Sigma) \otimes \mathfrak{g} \oplus  \Omega^0(\Sigma)  \otimes \mathfrak{g} \oplus  \Omega^2(\Sigma)  \otimes \mathfrak{g})$ be the similar subspace constructed from $[1,2] \times \Sigma \cup H_2$. Then it is well-known that $\Lambda_{\pm 1}(t)$ defines a pair of \emph{Lagrangian subspaces} in $L^2(\Omega^1(\Sigma) \otimes \mathfrak{g} \oplus  \Omega^0(\Sigma)  \otimes \mathfrak{g} \oplus  \Omega^2(\Sigma)  \otimes \mathfrak{g})$ with respect to the symplectic form $\tilde{\omega}_{\Sigma}$, see \cite[Section 2]{nicolaescu1995maslov} for example. The paths $t \mapsto \Lambda_{\pm}(t)$ are continuously differentiable in the sense that the paths of orthogonal projection operators $P_{\Lambda_{\pm}(t)}$ onto $\Lambda_{\pm}(t)$ are continuously differentiable.

\begin{Definition}
Let $L_1^2([-1,1] \times \Sigma; \Lambda_{-1}(t), \Lambda_{1}(t))$ be the subspace of $L_1^2([-1,1]; \Omega^1(\Sigma) \otimes \mathfrak{g} \oplus  \Omega^0(\Sigma)  \otimes \mathfrak{g} \oplus  \Omega^2(\Sigma)  \otimes \mathfrak{g})$ consisting of elements whose restriction to $\{ \pm 1 \} \times \Sigma$ lies in the space $\Lambda_{\pm 1}(t)$. Let $A_{s,t} \in \clA(\Sigma)$ be the family described above. Define
$$D(t): L_1^2([-1,1] \times \Sigma; \Lambda_{-1}(t), \Lambda_{1}(t)) \rightarrow L^2([-1,1]; \Omega^1(\Sigma) \otimes \mathfrak{g} \oplus  \Omega^0(\Sigma)  \otimes \mathfrak{g} \oplus  \Omega^2(\Sigma)  \otimes \mathfrak{g})$$
to be the operator given by
 $$
 D(t)
 \begin{pmatrix}
  a  \\
  \phi \\
  \psi
 \end{pmatrix}
 = 
 \begin{pmatrix}
* \dot a + d_{A_{s,t}}\phi + d_{A_{s,t}}^*\psi - *dX_s(A_{s,t})(a)\\
- * \dot \psi + d^*_{A_{s,t}} a\\
* \dot \phi + d_{A_{s,t}} a
 \end{pmatrix},
 $$
where the dot represents taking derivative along the $s$--direction.
\end{Definition}

\begin{remark}
The operator $D(t)$ could be written as $\tilde{J}_{\Sigma} (\frac{\partial}{\partial s} + S_{A_{s,t}} - dX_s(A_{s,t}))$. Using the Lagrangian boundary conditions $\Lambda_{\pm 1}(t)$, it is easy to see that $D(t)$ defines a self-adjoint operator such that the inclusion map from its domain to the background Hilbert space is a compact operator. Meanwhile, $D(t)$ is exactly the restriction of the operator $K_{B_t, {\bf q}}$ on the region $[-1,1] \times \Sigma \subset Y$.
\end{remark}

As a consequence, we obtain a $1$--parameter family of self-adjoint operators $\{ D(t) \}_{0 \leq t \leq 1}$. Although this family has varying domains $L_1^2([-1,1] \times \Sigma; \Lambda_{-1}(t), \Lambda_{1}(t))$, its spectral flow is nonetheless well-defined in this situation, see \cite[Appendix A]{salamon2008instanton}. Note that all the discussions above are compatible with the $\Stab(B_t)$--action, therefore the $H$--equivariant spectral flow $Sf_{H}(D(t))$ of the family $\{ D(t) \}_{0 \leq t \leq 1}$ is also well-defined. Now we are ready to state the main result of this subsection.

\begin{Theorem}
\label{thm_compare}
Let $\{ K_{B_t, f_{\bf q}} \}_{0 \leq t \leq 1}$ and $\{ D(t) \}_{0 \leq t \leq 1}$ be as above. Then
$$Sf_{H}(K_{B_t, f_{\bf q}}) = Sf_{H}(D(t)).$$
\end{Theorem}

The proof of Theorem \ref{thm_compare} is based on adapting the arguments in \cite{nicolaescu1995maslov}. The idea is to modify the path of self-adjoint operators to a preferred form to simplify the calculation of the spectral flow.

We start by reducing the calculation of the equivariant spectral flow to the calculation of the ordinary spectral flow. Identifying $(\Omega^0(Y) \oplus \Omega^1(Y)) \otimes \mathfrak{g}$ with the space of smooth sections of the vector bundle $\mathfrak{g} \oplus T^*Y \otimes \mathfrak{g}$, where $\mathfrak{g}$ stands for the product $\mathfrak{g}$--bundle over $Y$, for any connection $B \in \clC(Y)$, its stabilizer $\Stab(B) \subset G$ acts on $(\Omega^0(Y) \oplus \Omega^1(Y)) \otimes \mathfrak{g}$ by the fiberwise adjoint $G$--action on the $\mathfrak{g}$ factor. Choose $H \subset G$ as above. Suppose the adjoint action of $H$ on $\mathfrak{g}$ decomposes $\mathfrak{g}$ into isotypic pieces
$$\mathfrak{g} \cong V_{1}^{\oplus a_1} \oplus \cdots \oplus V_{m}^{\oplus a_m},$$
we define the vector bundle
$$\tilde{E}_{i} = (\Omega^0(Y) \oplus \Omega^1(Y)) \otimes V_{i}^{\oplus a_i}$$
and the vector bundle
$$E_i =(\Omega^0(Y) \oplus \Omega^1(Y)) \otimes \bK_i^{a_i}, $$
where $\bK_i = \Hom_{H}(V_i, V_i)$. Then the global sections of $E_i$ could be identified with 
$$\Hom_{H}(V_i, (\Omega^0(Y) \oplus \Omega^1(Y)) \otimes \mathfrak{g}).$$
Let $L^2_k(E_i)$ be the $L^2_k$--completion of the space of smooth sections of $E_i$ over $Y$. Then this encodes the information of the $V_i$-isotypic piece of the $H$--Banach space $L_k^2(\mathfrak{g} \oplus T^*Y \otimes \mathfrak{g})$. By the $H$--equivariance of the operator $K_{B_t, f_{\bf q}}$, it induces an operator
$$K^i_{B_t, f_{\bf q}} = L^2_k(E_i) \rightarrow L^2_{k-1}(E_i).$$

Correspondingly, we have
$$\Hom_{H}(V_i, L_k^2(\Omega^0([-1,1] \times \Sigma) \otimes \mathfrak{g}) \oplus L_k^2(\Omega^1([-1,1] \times \Sigma)  \otimes \mathfrak{g})) \cong L^2_k(E_{i}|_{[-1,1] \times \Sigma})$$
and the spaces over the Riemann surface $\Sigma$ given by
$$\Hom_{H}(V_i, L_k^2(\Omega^1(\Sigma) \otimes \mathfrak{g} \oplus  \Omega^0(\Sigma)  \otimes \mathfrak{g} \oplus  \Omega^2(\Sigma)  \otimes \mathfrak{g})) \cong L^2_k(E_{i}|_{\{0\} \times \Sigma}).$$
The symplectic form $\tilde{\omega}_{\Sigma}$ induces a symplectic form on $L^2(E_{i}|_{\{0\} \times \Sigma})$. By restricting $\Lambda_{\pm 1}(t)$ to the $V_i$--isotypic piece, we obtain Lagrangian subspaces
$$\Lambda^{i}_{\pm 1}(t) \subset L^2(E_{i}|_{\{0\} \times \Sigma}).$$
As a result, the $H$--equivariant operator $D(t)$ induces operators
$$D^{i}(t): L_1^2(E_{i}|_{[-1,1] \times \Sigma}; \Lambda^{i}_{-1}(t), \Lambda^{i}_{1}(t)) \rightarrow L^2(E_{i}|_{[-1,1] \times \Sigma})$$
for all $i = 1, \dots, m$. According to Definition \ref{def_equi_spectral_flow}, we have
\begin{Lemma}
\label{lem_nonequiv_flow}
$Sf_{H}(K_{B_t, f_{\bf q}}) = Sf_{H}(D(t))$ if and only if $Sf(K^i_{B_t, f_{\bf q}}) = Sf(D^{i}(t))$ for any $1 \leq i \leq m$. \qed
\end{Lemma}
From now on, we will focus on the operators $K^i_{B_t, f_{\bf q}}$ and $D^{i}(t)$ for a fixed $i \in \bZ \cap [1,m]$.

\begin{Definition}
An operator $D: L^2_k(E_i) \rightarrow L^2_{k-1}(E_i)$ is said to satisfy \emph{unique continuation} if for any $v \in \ker(D)$ such that $v|_{U} = 0$ where $U \subset H_2 \subset Y$ is a non-empty open set, we have $v \equiv 0$.
\end{Definition}

\begin{Lemma}
\label{lem_uniq_cont}
For any $t \in [0,1]$, the operator $K^i_{B_t, f_{\bf q}}$ satisfies unique continuation.
\end{Lemma}
\begin{proof}
Let $v \in \ker(K^i_{B_t, f_{\bf q}})$ and $v|_{U} = 0$ where $U$ is as above. The operator $K^i_{B_t, f_{\bf q}} |_{[1,2] \times \Sigma \cup H_2}$ comes from the twisted signature operator therefore is a Dirac operator. By the unique continuation property of Dirac operators, we know that $v$ is identically $0$ over $[1,2] \times \Sigma \cup H_2$. Over the region $[-1, 1] \times \Sigma$, the element $v$ defines a curve $v(s)$ in $L^2(E_{i}|_{\{0\} \times \Sigma})$ satisfying the ODE 
$$\tilde{J}_{\Sigma} (\frac{\partial}{\partial s} + S_{A_{s,t}} - dX_s(A_{s,t})) v(s) = 0.$$
Here we use the same notation to represent the induced operators on $L^2(E_{i}|_{[-1,1] \times \Sigma})$.
Since the solution is uniquely determined by the initial value and $v|_{\{1\} \times \Sigma} = 0$, we know that $v|_{[-1,1] \times \Sigma} = 0$. The vanishing of $v$ over $H_1 \cup [-2, -1] \times \Sigma$ comes again from the unique continuation property of the twisted signature operator on $H_1 \cup [-2, -1] \times \Sigma$ and the fact that we have already shown that $v_{\{-1\} \times \Sigma} = 0$. Therefore the lemma is proved.
\end{proof}

\begin{Lemma}
\label{lem_bound}
The operator $K^i_{B_1, f_{\bf q}} - K^i_{B_0, f_{\bf q}}$ is a bounded operator on $L^2_k(E_i)$.
\end{Lemma}
\begin{proof}
By definition, $K^i_{B_1, f_{\bf q}} - K^i_{B_0, f_{\bf q}}$ is the sum of a $0$--th order differential operator depending on $B_0, B_1$ and the induced operator on $L^2_k(E_i)$ given by the difference
$$DV_{f_{\bf q}}(B_1) - DV_{f_{\bf q}}(B_0).$$
Use the definition and \cite[Proposition 3.5]{kronheimer2011knot}, these $2$ terms are both bounded by $C (\|B_0\|_{L^2_k} + \|B_1\|_{L^2_k})$ where $C >0$ is some constant depending on $Y$, so the lemma is proved.
\end{proof}

Recall that for a family of self-adjoint operators $\{ D_t \}_{0 \leq t \leq 1}$, its \emph{resonance set} $Z(D_t)$ consists of $t \in [0,1]$ such that $\ker(D_t) \neq \{ 0 \}$. Suppose $t_0 \in Z(D_t)$ and let $P_{t_0}$ be the orthogonal projection onto $\ker(D_{t_0})$ . Then the \emph{resonance matrix} $R(D_{t_0})$ at $t_0$ is defined to be the map $P_{t_0} \dot{D}_{t_0}: \ker(D_{t_0}) \rightarrow \ker(D_{t_0})$. 

\begin{Definition}
A family of self-adjoint operators $\{ D_t \}_{0 \leq t \leq 1}$ is said to be \emph{positive} if the resonance matrices are all positive definite. It is said to be \emph{negative} if the family $\{ - D_t \}_{0 \leq t \leq 1}$ is positive.
\end{Definition}

Using Kato's selection theorem, if the family $\{ D_t \}_{0 \leq t \leq 1}$ is positive or negative, the set $Z(D_t)$ is discrete. 

\begin{Lemma}
\label{lem_definite}
The $1$--parameter family $\{ K^i_{B_t, f_{\bf q}} \}_{0 \leq t \leq 1}$ defined over $L^2_1(E_i) \subset L^2(E_i)$ is homotopic relative end points to the concatenation of a positive $1$--parameter family and a negative $1$--parameter family such that each operator in these $2$--families satisfies unique continuation.
\end{Lemma}
\begin{proof}
Using Lemma \ref{lem_bound}, we can choose some $C > 0$ such that the norm of $K^i_{B_1, f_{\bf q}} - K^i_{B_0, f_{\bf q}}$ over $L^2(E_i)$ is bounded above by $C -1$. Consider the families
$$K^{i}_{+}(t) = K^i_{B_0, f_{\bf q}} + t C \cdot \id$$
$$K^{i}_{-}(t) = K^i_{B_0, f_{\bf q}} + C \cdot \id + t (K^i_{B_1, f_{\bf q}} - K^i_{B_0, f_{\bf q}} - C \cdot \id)$$
for $0 \leq t \leq 1$. It is easy to see that $K^{i}_{+}(t)$ is positive and $K^{i}_{-}(t)$ is negative. The affine homotopy between $\{ K^i_{B_t, f_{\bf q}} \}_{0 \leq t \leq 1}$ and the concatenation of these $2$ families gives the desired homotopy. For the unique continuation property, one proceeds exactly the same as in the proof of Lemma \ref{lem_uniq_cont} by noticing that $K^{i}_{\pm}(t)$ could be described by Dirac operators on $H_1 \cup [-2,-1] \times \Sigma \cup [1,2] \times \Sigma \cup H_2$ and the kernel of $K^{i}_{\pm}(t)$ consists of solutions to an ODE on $[-1,1] \times \Sigma$.
\end{proof}

Given a positive or negative family $\{ D_t \}_{0 \leq t \leq 1}$, we wish to perturb it further so that for any $t_0 \in Z(D_t)$, the kernel of $D_{t_0}$ becomes $1$--dimensional. This is the transversality result established in \cite[Proposition 3.6]{nicolaescu1995maslov}. The unique continuation condition stated above guarantees that the proof of \cite[Proposition 3.6]{nicolaescu1995maslov} works without change to give the following statement.

\begin{Lemma}
\label{lem_ope_perturb}
For the two families $\{ K^{i}_{\pm}(t) \}_{0 \leq t \leq 1}$ constructed in Lemma \ref{lem_definite}, one can find 
$$\alpha_{\pm}: [0,1] \rightarrow C^{\infty}(\Hom(E_i, E_i))$$ which are arbitrarily small such that the following holds:
\begin{enumerate}
\item $\alpha_{\pm}(t)$ is supported away from $[-2, 2] \times \Sigma$;
\item The family $K^{i}_{+}(t) + \alpha_{+}(t)$ (resp. $K^{i}_{-}(t) + \alpha_{-}(t)$) remains to be positive (resp. negative) and for any $t_0 \in Z(K^{i}_{+}(t) + \alpha_{+}(t))$ (resp. $t_0 \in Z(K^{i}_{-}(t) + \alpha_{-}(t))$ ), the kernel of $K^{i}_{+}(t_0) + \alpha_{+}(t_0)$ (resp. $K^{i}_{-}(t_0) + \alpha_{-}(t_0)$) is actually $1$--dimensional. \qed
\end{enumerate} 
\end{Lemma}

\begin{proof}[Proof of Theorem \ref{thm_compare}]
Let us fix $\alpha_{\pm}(t)$ as in Lemma \ref{lem_ope_perturb} and denote $\tilde{K}^{i}_{\pm}(t) = K^{i}_{\pm}(t) + \alpha_{\pm}(t)$. By construction, $\{ \tilde{K}^{i}_{\pm}(t) \}_{0 \leq t \leq 1}$ defines operators
$$\tilde{K}^{i}_{\pm}(t) : L^2_1(E_i) \rightarrow L^2(E_i)$$
for all $0 \leq t \leq 1$ which are Dirac operators when restricted to $H_1 \cup [-2,-1] \times \Sigma \cup [1,2] \times \Sigma \cup H_2$ and they are cylindrical over $[-2,-1] \times \Sigma \cup [1,2] \times \Sigma$. Therefore, they define Lagrangian subspaces 
$$\tilde{\Lambda}^{i}_{-1}(t)_{\pm} \subset L^2(E_{i}|_{\{0\} \times \Sigma})$$
consisting of restriction of elements in $L^2_{1/2}(E_{i}|_{H_1 \cup [-2,-1] \times \Sigma})$ lying in 
$$\ker(\tilde{K}^{i}_{\pm}(t)|_{H_1 \cup [-2,-1] \times \Sigma}).$$
Similarly one can construct 
$$\tilde{\Lambda}^{i}_{1}(t)_{\pm} \subset L^2(E_{i}|_{\{0\} \times \Sigma})$$
by considering the restriction to $[1,2] \times \Sigma \cup H_2$.  We can then construct the operators
$$\tilde{D}^{i}_{\pm}(t): L_1^2(E_{i}|_{[-1,1] \times \Sigma}; \tilde{\Lambda}^{i}_{-1}(t)_{\pm}, \tilde{\Lambda}^{i}_{1}(t)_{\pm}) \rightarrow L^2(E_{i}|_{[-1,1] \times \Sigma})$$
similar to the construction of $D(t)$, which define self-adjoint operators using the Lagrangian boundary conditions.

Making $\alpha_{\pm}(t)$ smaller if necessary, the spectral flow of the family $\{ K_{B_t, f_{\bf q}} \}_{0 \leq t \leq 1}$ is the same as the sum of the spectral flows of $\{ \tilde{K}^{i}_{+}(t) \}_{0 \leq t \leq 1}$ and $\{ \tilde{K}^{i}_{-}(t) \}_{0 \leq t \leq 1}$. Similarly, by concatenating the Lagrangian boundary conditions, it is easy to see that the spectral flow of $\{ D^{i}(t) \}_{0 \leq t \leq 1}$ is equal to the sum of the spectral flows of $\{ \tilde{D}^{i}_{+}(t) \}_{0 \leq t \leq 1}$ and $\{ \tilde{D}^{i}_{-}(t) \}_{0 \leq t \leq 1}$. Using Lemma \ref{lem_nonequiv_flow}, it suffices to show that
$$Sf(\tilde{K}^{i}_{+}(t)) = Sf(\tilde{D}^{i}_{+}(t)) \mbox{ and } Sf(\tilde{K}^{i}_{-}(t)) = Sf(\tilde{D}^{i}_{-}(t)). $$

We prove the plus version of the above statement and the negative version holds using the same argument. From the definition, we see that
$$Z(\tilde{K}^{i}_{+}(t)) = Z(\tilde{D}^{i}_{+}(t))$$
therefore it suffices to show that the resonance matrices of $\{ \tilde{K}^{i}_{+}(t) \}_{0 \leq t \leq 1}$ and $\{ \tilde{D}^{i}_{+}(t) \}_{0 \leq t \leq 1}$ have the same sign at any $t_0 \in Z(\tilde{K}^{i}_{+}(t)) = Z(\tilde{D}^{i}_{+}(t))$. In other words, we just need to show that the resonance matrix of $\tilde{D}^{i}_{+}(t)$ at $t = t_0$ is positive definite. Let $v \in \ker(\tilde{D}^{i}_{+}(t_0)) - \{0\}$. Note that $v$ could also be viewed as an element of $\ker(\tilde{K}^{i}_{+}(t_0)).$ Let $\chi: [0,1] \rightarrow \bR$ be a non-negative smooth function supported in $(0,1)$ which is strictly positive over $[\frac14, \frac34]$. Then $\chi$ defines a function on $[0,1] \times \Sigma \subset Y$ by composing it with the projection to the $s$--coordinate. This further extends to a smooth function over $Y$ in the obvious way. Then for $\epsilon > 0$ sufficiently small, the family
\begin{equation}
\label{eqn_def_three}
\{ \tilde{K}^{i}_{+}(t_0) + (t - t_0) \chi \cdot \id \}_{t_0 - \epsilon \leq t \leq t_0 + \epsilon}
\end{equation}
is $C^0$--close to $\{ \tilde{K}^{i}_{+}(t) \}_{t_0 - \epsilon \leq t \leq t_0 + \epsilon}$ and the affine homotopy between them preserves the spectral flow. This is because $\langle v, \chi v \rangle_{L^2} > 0$ and the kernel is $1$--dimensional: if $\langle v, \chi v \rangle_{L^2} = 0$, the restriction of $v$ to the time slice $\{ \frac12 \} \times \Sigma$ is $0$ thus it vanishes on $[-1, 1] \times \Sigma$ by uniqueness of solutions to ODE therefore vanishes on whole $Y$ by unique continuation of Dirac operators. Accordingly, we can construct the induced affine homotopy from $\{ \tilde{D}^{i}_{+}(t) \}_{t_0 - \epsilon \leq t \leq t_0 + \epsilon}$ to 
\begin{equation}
\label{eqn_def_two}
\{ \tilde{D}^{i}_{+}(t_0) + (t - t_0) \chi \cdot \id \}_{t_0 - \epsilon \leq t \leq t_0 + \epsilon}
\end{equation}
preserving the spectral flow. Now notice that because $\chi$ is supported in the region $[0,1] \times \Sigma$ and the kernels of the new families at time $t_0$ are all spanned by $v$, we see that the resonance matrices of these $2$ families \eqref{eqn_def_three} and \eqref{eqn_def_two} have the same value at time $t_0$. This finishes the proof.
\end{proof}

\subsection{An adiabatic limit result}

Let $\{ A_{s,t} \}_{-1 \leq s \leq 1, 0 \leq t \leq 1} \subset A_{F}(\Sigma)$ be a smooth family of flat connections on $\Sigma$. Viewing $\Sigma$ as the boundary of $H_1 \cup [-2,-1] \times \Sigma$ and $[1,2] \times \Sigma \cup H_2$, we further suppose that $A_{-1, t}$ extends to a flat connection $B_{-1,t}$ on $H_1 \cup [-2,-1] \times \Sigma$ and $A_{1, t}$ extends to a flat connection $B_{1,t}$ on $[1,2] \times \Sigma \cup H_2$ for all $t \in [0,1]$. Then we can define $\{ \Lambda_{\pm 1}(t) \}_{0 \leq t \leq 1}$ be the same as in the previous subsection. Note that we do not require that $A_{s, 0}$ or $A_{s,1}$ comes from the restriction of a $f_{\bf q}$--perturbed flat connection on $Y$ as in the previous subsection. This setup gives us more flexibility for later calculations. Furthermore, we assume that $\Stab(A_{s,t}) \cong H \subset G$ remains to be the same for all $(s,t) \in [-1,1] \times [0,1]$.

Given $A_{s,t} \in \clA_{F}(\Sigma)$, let $\clH_{A_{s,t}}$ be the kernel of the twisted de Rham operator
$$S_{A_{s,t}}: L^2(\Omega^1(\Sigma) \otimes \mathfrak{g} \oplus  \Omega^0(\Sigma)  \otimes \mathfrak{g} \oplus  \Omega^2(\Sigma)  \otimes \mathfrak{g}) \rightarrow L^2(\Omega^1(\Sigma) \otimes \mathfrak{g} \oplus  \Omega^0(\Sigma)  \otimes \mathfrak{g} \oplus  \Omega^2(\Sigma)  \otimes \mathfrak{g}).$$
In other words, $\clH_{A_{s,t}}$ is given by the total space of twisted harmonic forms on $\Sigma$ with respect to the flat connection $A_{s,t}$. This is a symplectic vector space using the restriction of the symplectic form $\tilde{\omega}_{\Sigma}$. Let 
$$\pi_{A_{s,t}}: L^2(\Omega^1(\Sigma) \otimes \mathfrak{g} \oplus  \Omega^0(\Sigma)  \otimes \mathfrak{g} \oplus  \Omega^2(\Sigma)  \otimes \mathfrak{g}) \rightarrow \clH_{A_{s,t}}$$
be the $L^2$--orthogonal projection.
Define
$$\clL_{\pm 1}(t) := \clH_{A_{\pm 1,t}} \cap \Lambda_{\pm 1}(t), \mbox{ for } 0 \leq t \leq 1.$$

\begin{Lemma}
$\clL_{i}(t)$ is a Lagrangian subspace of $\clH_{A_{i,t}}$ for $i=\pm1$.
\end{Lemma}

\begin{proof}
By definition, $\clL_i(t)$ is isotropic. So we just need to prove that $\dim \clL_i(t) \geq \frac12 \dim \clH_{A_{i,t}}$. Notice that $(0,\ker d_{A_{i,t}}, 0)\subset \clL_i$ because this comes from the Lie algebra of the stabilizer group and $\Stab(B_{i,t}) = \Stab(A_{i,t})$.

Let $m = \dim \ker d_{B_{-1,t}}$. Recall that $h$ is the genus of the Riemann surface $\Sigma$. Fix a base point $p_0$ on $\Sigma \cong \partial H_1$, let $\gamma_1,\cdots,\gamma_h$ be $h$ closed curved based on $p_0$ that generates the fundamental group of $H_1$. Let $\hat \clC$ be the space of smooth connections $B$ on $H_1\cup [-2,-1]\times \Sigma$ such that $B(\partial_s) = 0$ on the boundary. Define
$$
\hol: \hat\clC \to G^h
$$
to be the map given by holonomies along $\gamma_i$. 
Let $U = (-\epsilon,\epsilon)^ {h}$, let $\Phi: U\to \hat \clC$ be a map such that $\Phi(0)$ equals $B_{-1,t}$, and that $\ima d(\hol\circ \Phi)$ is surjective at zero. By construction, $\ima d\Phi(0)\subset \Omega^1(H_1\cup[-2,-1]\times \Sigma) \otimes \mathfrak{g}$ is a  $h \cdot \dim G$ dimensional linear space. For every $u\in \ima d\Phi(0)$, there is a unique $v$ such that $v\in \ima d_{B_{-1,t}}$, $*v = 0$ on $\{-1\} \times \Sigma$, and $u':=u- d_{B_{-1,t}} v \in \ker(d_{B_{-1,t}}^*)$. The kernel of $\ima d\Phi(0)$ under the map $u\mapsto u'$ has dimension at most $m$ by considering the holonomies, therefore the image of $\ima d\Phi(0)$ under the map $u\mapsto u'$ has dimension at least $(h-1)\cdot \dim G+m$. Therefore 
$$
\clL_{-1}(t) \cap (\ker (d_{B_{-1,t}} +d^*_{B_{-1,t}}), 0, 0) \subset \clH_{A_{-1,t}}
$$
has dimension at least $(h-1)\cdot \dim G+m$, which proved that $\clL_{-1}(t)$ is Lagrangian. The same argument works to prove $\clL_{i}(t)$ is Lagrangian.
\end{proof}

By definition, $\dim \clH_{A_{s,t}}$ remains to be the same for all $(s,t) \in [-1,1] \times [0,1]$. For a given $t \in [0,1]$, consider the spaces
$$L_k^2([-1,1];\clH_{A_{s,t}}) \subset L^2_k([-1,1]; \Omega^1(\Sigma) \otimes \mathfrak{g} \oplus  \Omega^0(\Sigma)  \otimes \mathfrak{g} \oplus  \Omega^2(\Sigma)  \otimes \mathfrak{g})$$
$$L_k^2([-1,1];\clH_{A_{s,t}};\clL_{-1}(t), \clL_{1}(t)) \subset L_k^2([-1,1] \times \Sigma; \Lambda_{-1}(t), \Lambda_{1}(t))$$
consisting of elements $(a, \phi, \psi)$ whose value at each $s \in [-1,1]$ lies in $\clH_{A_{s,t}}$. Then we can define a self-adjoint operator 

\begin{equation}
\label{eqn_adia_operator}
D_0(t): L_1^2([-1,1];\clH_{A_{s,t}};\clL_{-1}(t), \clL_{1}(t)) \rightarrow L^2([-1,1];\clH_{A_{s,t}})
\end{equation}
given by
$$
 D_0(t)
  \begin{pmatrix}
  a  \\
  \phi \\
  \psi
 \end{pmatrix}
 = 
 \pi_{A_{s,t}} 
 \begin{pmatrix}
* \dot a  - *dX_s(A_{s,t})(a) \\
- * \dot \psi\\
 * \dot \phi 
 \end{pmatrix}.
 $$
By construction, the family $\{D_0(t)\}_{0 \leq t \leq 1}$ has a well-defined $H$--equivariant spectral flow $Sf_{H}(D_0(t))$. The following is another comparison result of spectral flows.

\begin{Theorem}
\label{thm_spec_comp_finite}
Let $\{ D(t) \}_{0 \leq t \leq 1}$ be the family constructed in Section \ref{subsec_spec_comp} and let $\{D_0(t)\}_{0 \leq t \leq 1}$ be the family constructed above. Then
$$Sf_{H}(D(t)) = Sf_{H}(D_0(t)).$$
\end{Theorem}

The proof of Theorem \ref{thm_spec_comp_finite} is exactly the same as the proof of \cite[Theorem 6.1]{dostoglou1994cauchy}. We will sketch a proof for completeness. Because we have used connections in temporal gauge, the covariant differentiation $\nabla_s$ in \cite{dostoglou1994cauchy} is replaced by the ordinary differentiation $\partial_s$. Note that the calculation of equivariant spectral flow could be reduced to the calculation of ordinary spectral flow as in Section \ref{subsec_spec_comp} so the $H$--equivariance will be suppressed in this subsection.

The idea is to use the $1$--parameter family of operators 
$$D_{\epsilon}(t): L_1^2([-1,1] \times \Sigma; \Lambda_{-1}(t), \Lambda_{1}(t)) \rightarrow L^2([-1,1]; \Omega^1(\Sigma) \otimes \mathfrak{g} \oplus  \Omega^0(\Sigma)  \otimes \mathfrak{g} \oplus  \Omega^2(\Sigma)  \otimes \mathfrak{g})$$
given by the formula
 $$
 D_{\epsilon}(t)
 \begin{pmatrix}
  a  \\
  \phi \\
  \psi
 \end{pmatrix}
 = 
 \begin{pmatrix}
* \dot a + d_{A_{s,t}}\phi + d_{A_{s,t}}^*\psi - *dX_s(A_{s,t})(a)\\
- * \dot \psi + \frac{1}{\epsilon^2} d^*_{A_{s,t}} a\\
* \dot \phi + \frac{1}{\epsilon^2} d_{A_{s,t}} a
 \end{pmatrix}
 $$
depending on the parameter $\epsilon > 0$, so that the spectral flow of $\{ D(t) \}_{0 \leq t \leq 1} = \{ D_{1}(t) \}_{0 \leq t \leq 1}$ could be reduced to the spectral flow of $\{D_0(t)\}_{0 \leq t \leq 1}$ by passing to the adiabatic limit $\epsilon \rightarrow 0$.

Fix $t \in [0,1]$. Introduce norms 
$$\|(a, \phi, \psi) \|_{0, \epsilon}^2 = \| a \|^2 + \epsilon^2 \| \phi \|^2 + \epsilon^2 \| \psi \|^2$$
on the space  $L^2([-1,1]; \Omega^1(\Sigma) \otimes \mathfrak{g} \oplus  \Omega^0(\Sigma)  \otimes \mathfrak{g} \oplus  \Omega^2(\Sigma)  \otimes \mathfrak{g})$ and 
\begin{align*}
\|(a, \phi, \psi) \|_{1, \epsilon}^2 &= \|a\|^2 + \|d_{A_{s,t}}a\|^2 + \|d_{A_{s,t}} * a\|^2 + \epsilon^2 \| \dot a \|^2 \\ 
& \quad + \epsilon^2\|d_{A_{s,t}} \phi \|^2 + \epsilon^4\| \dot \phi \|^2 + \epsilon^2\|d_{A_{s,t}}^* \psi \|^2 + \epsilon^4\| \dot \psi \|^2
\end{align*}
on $L^2_1([-1,1]; \Omega^1(\Sigma) \otimes \mathfrak{g} \oplus  \Omega^0(\Sigma)  \otimes \mathfrak{g} \oplus  \Omega^2(\Sigma)  \otimes \mathfrak{g})$ using the product metric on $[-1,1] \times \Sigma$.

\begin{proof}[Sketch of proof of Theorem \ref{thm_spec_comp_finite}] The proof breaks into five steps.

\emph{Step 1-Elliptic estimate:}There is an $\epsilon_0 > 0$ such that for any $0 < \epsilon < \epsilon_0$, there exists a constant $C > 0$ independent of $\epsilon$ such that
\begin{equation}
\label{eqn_ell_est}
\| u - \pi_{A_{s,t}}(u) \|_{1, \epsilon} \leq C \epsilon (\| D_{\epsilon}(t) u \|_{0, \epsilon} + \|\pi_{A_{s,t}}(u) \|_{L^2})
\end{equation}
for $u \in L_1^2([-1,1] \times \Sigma; \Lambda_{-1}(t), \Lambda_{1}(t))$. 

This is \cite[Lemma 7.3]{dostoglou1994cauchy}. Note that although our $3$--manifold $[-1,1] \times \Sigma$ has boundary, the Lagrangian boundary condition guarantees that the integration by parts formula entering the proof works without change. Meanwhile, our projection map $\pi_{A_{s,t}}$ involves the $H^0$ and $H^2$ part so the reducible connections would not affect the estimate.

\emph{Step 2-Convergence of resolvent set:} For every $C_0 > 0$ there exists constants $\epsilon_0 > 0$ and $C > 0$ such that the following holds for all $t \in [0,1]$ and $|\lambda| \leq C_0$. If
$$\| u_0 \|_{L^2} \leq C_0 \| D_{0}(t) u_0 - \lambda u_0 \|_{L^2}$$ 
for all $u_0 \in L_1^2([-1,1];\clH_{A_{s,t}};\clL_{-1}(t), \clL_{1}(t)),$
then
$$\| u \|_{0, \epsilon} \leq C \| D_{\epsilon}(t) u - \lambda u\|_{0, \epsilon}$$
for $0 < \epsilon < \epsilon_0$ and $u \in L_1^2([-1,1] \times \Sigma; \Lambda_{-1}(t), \Lambda_{1}(t))$.

This is \cite[Lemma 7.4]{dostoglou1994cauchy}. Note that the proof is based on a direct computation and an application of equation \eqref{eqn_ell_est} therefore it works in our situation without change.

\emph{Step 3-Refined convergence of operators:} Define
$$R_{0} = \{ (t, \lambda) \in [0,1] \times \bC \big| \lambda \mbox{ is not an eigenvalue of } D_0(t) \}$$
to be the resolvent set of the family $\{ D_{0}(t) \}_{0 \leq t \leq 1}$. Denote by $R_{\epsilon}$ is resolvent set of the family $\{ D_{\epsilon}(t) \}_{0 \leq t \leq 1}$. Then for every compact subset $K \subset R_0$, there exists a constant $\epsilon_0 > 0$ such that $K \subset R_{\epsilon}$ for $0 < \epsilon < \epsilon_0$ and 
\begin{equation}
\| \pi_{A_{s,t}}((\lambda \cdot \id - D_{\epsilon}(t))^{-1} v) - (\lambda \cdot \id - D_{0}(t))^{-1}\pi_{A_{s,t}} \|_{L^2} \leq C \epsilon \|v\|_{0, \epsilon}
\end{equation}
for $(t, \lambda) \in K$ and $v \in L^2([-1,1]; \Omega^1(\Sigma) \otimes \mathfrak{g} \oplus  \Omega^0(\Sigma)  \otimes \mathfrak{g} \oplus  \Omega^2(\Sigma)  \otimes \mathfrak{g})$.

This is \cite[Lemma 7.5]{dostoglou1994cauchy}.

\emph{Step 4-Multiplicity estimate:} For all $C_0 > 0$ there exists a constant $k_0 > 0$ such that the following holds for all $t \in [0,1]$ and $0 < \epsilon^2 + \delta^2 < k_0$. If $\lambda_0$ is an eigenvalue of $D_{0}(t)$ of multiplicity $m_0$ and 
$$u_0 \perp \ker(\lambda_0 \cdot \id - D_{0}(t)) \Longrightarrow \| u_0 \|_{L^2} \leq C_0 \| \lambda_0 u_0 - D_{0}(t) u_0 \|_{L^2}$$
for all $u \in L_1^2([-1,1];\clH_{A_{s,t}};\clL_{-1}(t), \clL_{1}(t))$, then the multiplicity of all eigenvalues of $\lambda$ of $D_{\epsilon}(t)$ with $|\lambda - \lambda_{0}| \leq \delta$ does not exceed $m_0$.

This follows from \cite[Lemma 7.6]{dostoglou1994cauchy}.

\emph{Step 5-Concluding the proof.} By Kato's selection theorem, we can find $\delta \in \bR$ which is arbitrarily close to $0$ such that the resonance set of the family of the operators $\{ D_0(t) - \delta \cdot \id \}_{0 \leq t \leq 1}$, written as $Z(D_0(t) - \delta \cdot \id)$ is finite and contained in $(0,1)$. Moreover, we can guarantee that the resonance matrices of this family is non-degenerate. Use $\sigma(D)$ to represent the spectrum of $D$. Let $Z(D_0(t) - \delta \cdot \id) = \{ t_1, \dots, t_N \}$. Recall that $R(D_0(t_i) - \delta \cdot \id)$ for $i = 1, \dots, N$ is the resonance matrix. Define
$$m_i := \mbox{sign}(R(D_0(t_i) - \delta \cdot \id)), i=1, \dots, N$$
to be the signatures of the resonance matrices. Then
$$Sf(\{D_0(t)\}_{0 \leq t \leq 1}) = \sum_{i=1}^{N} m_i$$
and $m_i$ could be computed as follows. Choose $\kappa > 0$ such that $\lambda = \delta$ is the only eigenvalue of $D_0(t_i)$ in the interval $[\delta - \kappa, \delta +\kappa]$. Now choose $\tau > 0$ such that $\delta \pm \kappa \notin \sigma(D_0(t))$ for all $t \in [t_i - \tau, t_i + \tau]$. Then
\begin{align*}
m_i &= \# \{ \lambda \in \sigma(D_0(t_i - \tau)) \big| \delta - \kappa < \lambda < \delta \} \\
& \quad - \# \{ \lambda \in \sigma(D_0(t_i + \tau)) \big| \delta - \kappa < \lambda < \delta \}.
\end{align*}
By \emph{Step 2}, there exists a constant $\epsilon_0 > 0$ such that $\delta \notin \sigma(D_{\epsilon}(t))$ for $t = t_i - \tau \mbox{ or } t_i + \tau)$ and $\delta \pm \kappa \notin \sigma(D_{\epsilon}(t))$ for all $t \in (t_i - \tau, t_i + \tau)$ as long as $0 < \epsilon < \epsilon_0$. Using \emph{Step 3}, \emph{Step 4} and the spectral projection operators, one can show that
\begin{align*}
m_i &= \# \{ \lambda \in \sigma(D_{\epsilon}(t_i - \tau)) \big| \delta - \kappa < \lambda < \delta \} \\
& \quad - \# \{ \lambda \in \sigma(D_{\epsilon}(t_i + \tau)) \big| \delta - \kappa < \lambda < \delta \}.
\end{align*}
Therefore we conclude that
$$Sf( D_{\epsilon}(t)) = \sum_{i=1}^{N} m_i = Sf(D_{0}(t))$$
for $\epsilon > 0$ sufficiently small. By homotopy invariance, $Sf(D_{\epsilon}(t))$ is independent of $\epsilon > 0$. This concludes the proof. For the full details of the proof, we refer the readers to \cite[Section 7]{dostoglou1994cauchy}.
\end{proof}

\subsection{From Maslov index to spectral flow and back}Throughout this subsection, the structure group $G $ is equal to $\SU(n)$ and $Y$ is an integer homology $3$--sphere. Furthermore, suppose that the genus of the Heegaard surface $\Sigma$ is $h \geq 3$. Let $H_{\hat{s}}$ and $f_{{\bf q}}$ be a compatible pair of small non-degenerate perturbations. Suppose $B \in \clC(Y)$ is a $f_{{\bf q}}$--perturbed flat connection of type $\sigma = ((n_1, m_1), \dots, (n_r, m_r)) \in \Sigma_{n}$ (see the end of Section \ref{sec_preliminary}). Without loss of generality, we can assume that $B$ is in temporal gauge over $[-2, 2] \times \Sigma$ and it is represented by a path $A_{s}:[-2,2] \rightarrow \clA^{0}_{F}(\Sigma)$ which is constant for $s \in [-2,-1] \cup [1,2]$. After applying a gauge transform, the connection $B$ is given by the direct sum of irreducible $\SU(n_i)$-connections $B^{(i)}$ on $E(n_i)$ where 
\begin{equation}
\label{eqn_decomp_last}
P \times_{\SU(n)} \bC^n \cong E(n_1)^{\oplus m_1} \oplus \cdots \oplus E(n_r)^{\oplus m_r}.
\end{equation}
Let $A^{(i)}_{s}$ be the induced path on the $E(n_i)$--component. Then $A^{(i)}_{-1}$ extends to a flat connection $B^{(i)}_{-1}$ on $H_1 \cup [-2,-1] \times \Sigma$ and $A^{(i)}_{1}$ extends to a flat connection $B^{(i)}_{1}$ on $[1,2] \times \Sigma \cup H_2$. Note that because $Y$ is an integer homology sphere, the product connection $\theta$ is a non-degenerate critical point of the unperturbed functional $\CS$. Because $f_{{\bf q}}$ is a small perturbation, there exists a $f_{{\bf q}}$--perturbed flat connection $\tilde{\theta}$ in temporal gauge over $[-2,2] \times \Sigma$ with stabilizer $G$ which lies in a contractible neighborhood of the product connection $\theta \in \clC(Y)$. Denote by $\tilde{\theta}^{(i)}$ the $E(n_i)$--component of $\tilde{\theta}$ under the decomposition \eqref{eqn_decomp_last} and $\tilde{\theta}_{s}$ is the induced path of flat connections in $\clA_{F}(\Sigma)$, which can be further assumed to lie in $\clA_{F}^0(\Sigma)$. Then we can find a path of connections
$$A^{(i)}_{-1,t}: [0,1] \rightarrow \clA_{F}^0(\Sigma) \mbox{ with } A^{(i)}_{-1,0} = \tilde{\theta}^{(i)}_{-1} \mbox{ and } A^{(i)}_{-1,1} = A^{(i)}_{-1}$$
such that $A^{(i)}_{-1,t}$ extends to a flat connection $B^{(i)}_{-1,t}$ on $E(n_i)$ over $H_1 \cup [-2,-1] \times \Sigma$ for all $t \in [0,1]$ which is irreducible expect for $t = 0$. Similarly, we can find the family $A^{(i)}_{1,t}$ on $[1,2] \times \Sigma \cup H_2$ and $B^{(i)}_{1,t}$ with the same properties. Then the union of paths $\{ A^{(i)}_{-1,t} \}_{0 \leq t \leq 1}, \{ \tilde{\theta}^{(i)}_{s} \}_{-1 \leq s \leq 1}, \{ A^{(i)}_{1,t} \}_{0 \leq t \leq 1}$ and $\{ A^{(i)}_{s} \}_{-1 \leq s \leq 1}$ defines a loop inside the space of flat connections on $E(n_i)$, denoted by $\clA_{F}^{(i)}(\Sigma)$. By our assumption that the genus of $\Sigma$ is at least $3$, \cite[Corollary 2.7]{daskalopoulos1995application} shows that the subset of $\clA_{F}^{(i)}(\Sigma)$ consisting of irreducible connections is simply-connected. Consequently, the loop constructed above could be extended to a smooth $2$--parameter family
$$A^{(i)}_{s,t}: [-1,1] \times [0,1] \rightarrow \clA_{F}^{(i)}(\Sigma)$$
such that $A^{(i)}_{s,t}$ is irreducible except for $(s,t) \in [-1,1] \times \{0\}$. Take the direct sum of these families for $1 \leq i \leq m$ and apply further gauge transformations if necessary, we obtain a $2$--parameter family
$$A_{s,t}: [-1,1] \times [0,1] \rightarrow \clA_{F}^{0}(\Sigma) \subset \clA_{F}(\Sigma).$$ 
Let $\tilde{\theta} = B_0$ and $B = B_1$, we can use $A_{s,t}$ to construct a family $\{ B_t \}_{0 \leq t \leq 1}$ satisfying the assumptions in the beginning of Subsection \ref{subsec_spec_comp}. According to this construction, $\Stab(A_{s,t})$ remains invariant for $(s,t) \in [-1,1] \times (0,1]$.

\begin{Lemma}
\label{lem_CSinv}
Suppose the perturbations $H_{\hat{s}}$ and $f_{\bf q}$ are $0$. Then for each $1 \leq i \leq r$, the symplectic area of $A^{(i)}_{s,t}: [-1,1] \times [0,1] \rightarrow \clA_{F}^{(i)}(\Sigma)$ under the Atiyah-Bott symplectic form $\omega_{\Sigma}$ is one-half of the Chern-Simons invariant of $B^{(i)}$.
\end{Lemma}
\begin{proof}
According to our convention of Chern-Simons functional \eqref{eqn_def_chern_simons},
$$
\begin{aligned}
\CS(B^{(i)}) &= \CS(B^{(i)}|_{H_1 \cup [-2,-1] \times \Sigma}) + \CS(B^{(i)}|_{[-1,1] \times \Sigma}) + \CS(B^{(i)}|_{[1,2] \times \Sigma \cup H_2}) \\
&= \int_{[0,1] \times [-1,1] \times \Sigma} \Tr(F_{B_t} \wedge F_{B_t}) \\
&= \int_{[0,1] \times [-1,1] \times \Sigma} \Tr((ds \wedge \partial_{s}A^{(i)}_{s,t} + dt \wedge \partial_{t}A^{(i)}_{s,t} + F_{A^{(i)}_{s,t}})^2 ) \\
&= 2 \int_{[0,1] \times [-1,1] \times \Sigma} \Tr(\partial_{s}A^{(i)}_{s,t} \wedge \partial_{t}A^{(i)}_{s,t}) dsdt.
\end{aligned}
$$
The second equality is an application of the Stokes' formula over the $4$--manifolds $[0,1] \times [-1,1] \times \Sigma, [0,1] \times (H_1 \cup [-2,-1] \times \Sigma), [0,1] \times ([1,2] \times \Sigma \cup H_2)$ and notice that $B^{(i)}|_{H_1 \cup [-2,-1] \times \Sigma}, B^{(i)}|_{[1,2] \times \Sigma \cup H_2}$ are connected to the product connection through flat connections. The above calculation proves the lemma.
\end{proof}

Recall that $L_1$ and $L_2$ are Lagrangians in the extended moduli space $\hat{\mathcal{M}}^{\mathfrak{g}}(\Sigma')$ constructed from handlebodies. The $f_{\bf q}$--perturbed flat connection $B$ is associated with the $G$--orbit $[A_1] \in \Phi_{H_{\hat{s}}}(L_1) \cap L_2$ in $\hat{\mathcal{M}}^{\mathfrak{g}}(\Sigma')$ in Proposition \ref{prop_set}. Let $\Phi_{H_{\hat{s}}}(\eta)$ be the time-$\eta$--flow of the Hamiltonian vector field $X_{H_{\hat{s}}}$. Then the composition
$$u(A_{s,t}) := \Phi_{H_{\hat{s}}}(1 - \eta)([A_{2 \eta - 1, t}]): (\eta, t) \in [0,1] \times [0,1] \rightarrow \hat{\mathcal{M}}^{\mathfrak{g}}(\Sigma')$$
maps $\{0\} \times [0,1]$ to $\Phi_{H_{\hat{s}}}(L_1)$ and maps $\{1\} \times [0,1]$ to $L_2$ and the stabilizer of every point in the image contains $H = \Stab(B) \cong \Stab([A_1])$ as a subgroup. Using Definition \ref{def_equi_mas}, $u(A_{s,t})$ has a well-defined $H$--equivariant Maslov index. Here is the precise statement of Theorem \ref{thm_index_intro}.

\begin{Theorem}
\label{thm_main}
$Sf_{H}(B, f_{\bf q}) - [\ker d_B] = \mu^{H}(u(A_{s,t}))$.
\end{Theorem}
\begin{proof}
By definition, the equivariant Maslov index $\mu^{H}(u(A_{s,t}))$ could be computed using the equivariant spectral flow of the operator $- J \frac{\partial}{\partial s}$ on the space of $L_1^2$--sections of the symplectic vector bundle
$$u(A_{s,t})^{*} (T \hat{\mathcal{M}}^{\mathfrak{g}}(\Sigma'))$$
with Lagrangian boundary conditions induced by $T \Phi_{H_{\hat{s}}}(L_1)$ and $T L_2$. Apply the fiberwise symplectomorphism induced by linearizing $\Phi_{H_{\hat{s}}}(1 - \eta)$, use the identification \eqref{cor_slice} and use the fact that the intersection $[\tilde{\theta}]$ is non-degenerate, we can see that $\mu^{H}(u(A_{s,t}))$ is given by the $H$--equivariant spectral flow of the family
$$ - D_0(t) : L_1^2([-1,1];\clH_{A_{s,t}} \cap H^1;\clL_{-1}(t), \clL_{1}(t)) \rightarrow L^2([-1,1];\clH_{A_{s,t}} \cap H^1)$$
from \eqref{eqn_adia_operator} where $t \in [\epsilon, 1]$ such that $\epsilon > 0$ is sufficiently small. The symbol $\cap H^1$ means that we require that the $0$-form and $2$--form components are $0$. This is exactly where the shifting $[\ker d_B] $ comes in. By Theorem \ref{thm_spec_comp_finite}, we see that $\mu^{H}(u(A_{s,t}))$ is equal to the $H$--equivariant spectral flow of the family $\{ - K_{B_t, f_{\bf q}} \}_{\epsilon \leq t \leq 1}$ up to a shifting by $[\ker d_B]$. By concatenating $\{ K_{B_t, f_{\bf q}}\}_{0 \leq t \leq 1}$ with a linear path between $K_{\tilde{\theta}, f_{\bf q}}$ and $K_{\theta, 0}$, using Theorem \ref{thm_compare}, we see that $\mu^{H}(u(A_{s,t}))$ exactly computes $Sf_{H}(B, f_{\bf q})$ (recall that this is defined using the linear path from $(B, f_{\bf q})$ to $(\theta, 0)$) by homotopy invariance.
\end{proof}

Therefore, we have the following immediate corollary from Definition \ref{defn_ind}, Lemma \ref{lem_CSinv} and Theorem \ref{thm_main}:
\begin{Corollary}
$\ind(B, f_{\bf q}) = \mu^{H}(u(A_{s,t})) - \sum_{n_i \geq 2} \frac{\langle \omega, u(\hat{A}^{(i)}_{s,t}) \rangle}{2 \pi^2 n_i} \cdot \tau_i$, where $\langle \omega, u(\hat{A}^{(i)}_{s,t}) \rangle$ is the symplectic area of the family $u(\hat{A}^{(i)}_{s,t})$ induced from the $i$--th component of the decomposition of a genuine flat $\hat{B}$ near $B$. \qed
\end{Corollary}

\begin{proof} [Proof of Theorem \ref{thm_BH_intro}]
Let $H_{\hat{s}}$ and $f_{\bf q}$ be a compatible pair. For any $G$--orbits of $\Phi_{H_{\hat{s}}}(L_1) \cap L_2$, without loss of generality we can assume that the disc $D(p)$ comes from a family $A_{s,t}(p) \subset \clA_{F}^{0}(\Sigma)$ as in this section. If $p$ corresponds to an irreducible $\SU(3)$--connection $B(p)$, by Theorem \ref{thm_main} the spectral flow from $K_{B(p), f_{\bf q}}$ is equal to $\mu(D(p))$ by noticing that the group acts trivially here and  $[\ker d_B] = 0$. If $p$ corresponds to a reducible $\SU(3)$--connection $B(p)$, it has stabilizer $\U(1)$ and $B(p)$ is gauge-equivalent to an $\SU(2)$--connection. The equivariant Maslov index $\mu^{H}(u(A_{s,t}))$ is the linear combination of the trivial representation and the weight $(-2)$--representation of $\U(1)$ in the representation ring of $\U(1)$. The coefficient before each of them, is given by $\mu_t(D(p))$ and $\mu_{n}(D(p))$ respectively. Then the theorem follows from checking the formula \eqref{thm_BH_intro} term-wisely with the formula in \cite[Theorem 1]{boden1998the}. Note that our definition of the equivariant spectral flow cancels out the factor $\frac12$ in \emph{loc.cit.} and our convention of spectral flow adds the term $[\ker d_B]$. Although the formula of Boden-Herald is written down using a holonomy perturbation which is not of the preferred form, the identification between \eqref{thm_BH_intro} and $\lambda_{\SU(3)}$ results from the independence of $\lambda_{\SU(3)}$ on holonomy perturbations, established in \cite{bai2020equivariant} and Section \ref{sec_preliminary}. Therefore the theorem is proved.
\end{proof}

%% file: main.bbl
\begin{bibdiv}
\begin{biblist}

\bib{abouzaid2017sheaf}{article}{
      author={Abouzaid, Mohammed},
      author={Manolescu, Ciprian},
       title={A sheaf-theoretic model for {$\SL (2,\mathbb{C})$} {F}loer
  homology},
        date={2017},
     journal={arXiv preprint arXiv:1708.00289},
}

\bib{akbulut1990casson}{article}{
      author={Akbulut, Selman},
      author={McCarthy, John~D},
       title={Casson's invariant for oriented homology 3-spheres: an
  exposition},
        date={1990},
}

\bib{boden1998the}{article}{
      author={Boden, Hans~U},
      author={Herald, Christopher~M},
       title={The {$\SU(3)$} {C}asson invariant for integral homology
  3-spheres},
        date={1998},
     journal={Journal of Differential Geometry},
      volume={50},
      number={1},
       pages={147\ndash 206},
}

\bib{BHK2001}{article}{
      author={Boden, Hans~U},
      author={Herald, Christopher~M},
      author={Kirk, Paul},
       title={An integer valued {$\SU(3)$} {C}asson invariant},
        date={2001},
     journal={Mathematical Research Letters},
      volume={8},
      number={5},
       pages={589\ndash 603},
}

\bib{boden2005integer}{article}{
      author={Boden, Hans~U},
      author={Herald, Christopher~M},
      author={Kirk, Paul~A},
       title={The integer valued {$\SU (3)$} casson invariant for brieskorn
  spheres},
        date={2005},
     journal={Journal of Differential Geometry},
      volume={71},
      number={1},
       pages={23\ndash 83},
}

\bib{boyer1990varieties}{article}{
      author={Boyer, Steven},
      author={Nicas, Andrew},
       title={Varieties of group representations and {C}asson's invariant for
  rational homology 3-spheres},
        date={1990},
     journal={Transactions of the American Mathematical Society},
      volume={322},
      number={2},
       pages={507\ndash 522},
}

\bib{bai2020equivariant}{article}{
      author={Bai, Shaoyun},
      author={Zhang, Boyu},
       title={Equivariant {C}erf theory and perturbative {$\SU (n)$} {C}asson
  invariants},
        date={2020},
     journal={arXiv preprint arXiv:2009.01118},
}

\bib{cappell2002perturbative}{article}{
      author={Cappell, Sylvain~E},
      author={Lee, Ronnie},
      author={Miller, Edward~Y},
       title={A perturbative {$\SU (3)$} {C}asson invariant},
        date={2002},
     journal={Commentarii Mathematici Helvetici},
      volume={77},
      number={3},
       pages={491\ndash 523},
}

\bib{cappell1990symplectic}{article}{
      author={Cappell, Sylvain~E},
      author={Lee, Ronnie},
      author={Miller, Edward~Y},
       title={A symplectic geometry approach to generalized {C}asson's
  invariants of 3-manifolds},
        date={1990},
     journal={Bulletin of the American Mathematical Society},
      volume={22},
      number={2},
       pages={269\ndash 275},
}

\bib{cappell1994maslov}{article}{
      author={Cappell, Sylvain~E},
      author={Lee, Ronnie},
      author={Miller, Edward~Y},
       title={On the maslov index},
        date={1994},
     journal={Communications on Pure and Applied Mathematics},
      volume={47},
      number={2},
       pages={121\ndash 186},
}

\bib{curtis1994generalized}{article}{
      author={Curtis, Cynthia~L},
       title={Generalized casson invariants for {$\SO(3), \U(2), \Spin(4)$} and
  {$\SO(4)$}},
        date={1994},
     journal={Transactions of the American Mathematical Society},
       pages={49\ndash 86},
}

\bib{donaldson1987orientation}{article}{
      author={Donaldson, Simon~K},
       title={The orientation of {Y}ang-{M}ills moduli spaces and 4-manifold
  topology},
        date={1987},
     journal={Journal of Differential Geometry},
      volume={26},
      number={3},
       pages={397\ndash 428},
}

\bib{dostoglou1994cauchy}{incollection}{
      author={Dostoglou, Stamatis},
      author={Salamon, Dietmar~A.},
       title={Cauchy-{R}iemann operators, self-duality, and the spectral flow},
        date={1994},
   booktitle={First {E}uropean {C}ongress of {M}athematics, {V}ol. {I}
  ({P}aris, 1992)},
      series={Progr. Math.},
      volume={119},
   publisher={Birkh\"{a}user, Basel},
       pages={511\ndash 545},
      review={\MR{1341835}},
}

\bib{daskalopoulos1995application}{article}{
      author={Daskalopoulos, Georgios~D},
      author={Uhlenbeck, Karen~K},
       title={An application of transversality to the topology of the moduli
  space of stable bundles},
        date={1995},
     journal={Topology},
      volume={34},
      number={1},
       pages={203\ndash 215},
}

\bib{floer1988instanton}{article}{
      author={Floer, Andreas},
       title={An instanton-invariant for 3-manifolds},
        date={1988},
     journal={Communications in mathematical physics},
      volume={118},
      number={2},
       pages={215\ndash 240},
}

\bib{huebschmann1995symplectic}{article}{
      author={Huebschmann, Johannes},
       title={Symplectic and {P}oisson structures of certain moduli spaces.
  {I}},
        date={1995},
        ISSN={0012-7094},
     journal={Duke Math. J.},
      volume={80},
      number={3},
       pages={737\ndash 756},
         url={https://doi.org/10.1215/S0012-7094-95-08024-7},
      review={\MR{1370113}},
}

\bib{jeffrey1994extended}{article}{
      author={Jeffrey, Lisa~C.},
       title={Extended moduli spaces of flat connections on {R}iemann
  surfaces},
        date={1994},
        ISSN={0025-5831},
     journal={Math. Ann.},
      volume={298},
      number={4},
       pages={667\ndash 692},
         url={https://doi.org/10.1007/BF01459756},
      review={\MR{1268599}},
}

\bib{kronheimer2011knot}{article}{
      author={Kronheimer, Peter~B},
      author={Mrowka, Tomasz~S},
       title={Knot homology groups from instantons},
        date={2011},
     journal={Journal of Topology},
      volume={4},
      number={4},
       pages={835\ndash 918},
}

\bib{lescop1996global}{book}{
      author={Lescop, Christine},
       title={Global surgery formula for the casson-walker invariant},
   publisher={Princeton University Press},
        date={1996},
}

\bib{marin1988nouvel}{article}{
      author={Marin, Alexis},
       title={Un nouvel invariant pour les spheres d'homologie de dimension
  trois},
        date={1988},
     journal={Seminare Bourbaki, 1987--88, Asterisque N},
      volume={693},
       pages={151},
}

\bib{mcduff2017introduction}{book}{
      author={McDuff, Dusa},
      author={Salamon, Dietmar},
       title={Introduction to symplectic topology},
   publisher={Oxford University Press},
        date={2017},
}

\bib{manolescu2012floer}{incollection}{
      author={Manolescu, Ciprian},
      author={Woodward, Christopher},
       title={Floer homology on the extended moduli space},
        date={2012},
   booktitle={Perspectives in analysis, geometry, and topology},
   publisher={Springer},
       pages={283\ndash 329},
}

\bib{nicolaescu1995maslov}{article}{
      author={Nicolaescu, Liviu~I.},
       title={The {M}aslov index, the spectral flow, and decompositions of
  manifolds},
        date={1995},
        ISSN={0012-7094},
     journal={Duke Math. J.},
      volume={80},
      number={2},
       pages={485\ndash 533},
         url={https://doi.org/10.1215/S0012-7094-95-08018-1},
      review={\MR{1369400}},
}

\bib{pozniak1999floer}{inproceedings}{
      author={Pozniak, Marcin},
       title={Floer homology, novikov rings and clean intersections},
        date={1999},
   booktitle={Northern california symplectic geometry seminar},
      volume={196},
       pages={119\ndash 181},
}

\bib{saveliev2011lectures}{book}{
      author={Saveliev, Nikolai},
       title={Lectures on the topology of 3-manifolds: an introduction to the
  {C}asson invariant},
   publisher={Walter de Gruyter},
        date={2011},
}

\bib{sjamaar1991stratified}{article}{
      author={Sjamaar, Reyer},
      author={Lerman, Eugene},
       title={Stratified symplectic spaces and reduction},
        date={1991},
     journal={Annals of Mathematics},
       pages={375\ndash 422},
}

\bib{salamon2008instanton}{article}{
      author={Salamon, Dietmar},
      author={Wehrheim, Katrin},
       title={Instanton floer homology with lagrangian boundary conditions},
        date={2008},
     journal={Geometry \& Topology},
      volume={12},
      number={2},
       pages={747\ndash 918},
}

\bib{taubes1990casson}{article}{
      author={Taubes, Clifford~H},
       title={Casson's invariant and gauge theory},
        date={1990},
     journal={Journal of Differential Geometry},
      volume={31},
      number={2},
       pages={547\ndash 599},
}

\bib{walker1992extension}{book}{
      author={Walker, Kevin},
       title={An extension of {C}asson's invariant},
   publisher={Princeton University Press},
        date={1992},
      number={126},
}

\end{biblist}
\end{bibdiv}
